\documentclass[11pt,oneside]{amsart}

      \usepackage{amssymb}
      \usepackage{amsmath}
      \usepackage{enumerate}   
 \numberwithin{equation}{section}

      \theoremstyle{plain}
      \newtheorem{theorem}{Theorem}[section]
      \newtheorem{lemma}[theorem]{Lemma}
      \newtheorem{corollary}[theorem]{Corollary}
      \newtheorem{proposition}[theorem]{Proposition}

      \theoremstyle{definition}
      \newtheorem{definition}[theorem]{Definition}
      
      \theoremstyle{remark}

      \newcommand{\R}{{\mathbb R}}

      \makeatletter
      \def\@setcopyright{}
      \def\serieslogo@{}
      \makeatother
   
\begin{document}

\title{A non-local one-phase free boundary problem from obstacle to cavitation
}


   \author{Yijing Wu}
   \address{Department of Mathematics, the University of Texas at Austin}
   \email{yijingwu@math.utexas.edu}
\maketitle

\begin{abstract}
   We consider a one-phase free boundary problem of the minimizer of the energy 
\[
J_{\gamma}(u)=\frac{1}{2}\int_{(B_1^{n+1})^+}{y^{1-2s}|\nabla u(x,y)|^2dxdy}+\int_{B_1^{n}\times \{y=0\}}{u^{\gamma}dx},
\]
with constants $0<s,\gamma<1$. It is an intermediate case of the fractional cavitation problem (as $\gamma=0$) and the fractional obstacle problem (as $\gamma=1$). We prove that the blow-up near every free boundary point is homogeneous of degree $\beta=\frac{2s}{2-\gamma}$, and flat free boundary is $C^{1,\theta}$ when $\gamma$ is close to 0.
\end{abstract}

\section{Introduction}
In this paper, we prove free boundary properties for minimizers of the following energy
\[
J_{\gamma}(u)=\frac{1}{2}\int_{(B_1^{n+1})^+}{y^{1-2s}|\nabla u(x,y)|^2dxdy}+\int_{B_1^{n}\times \{y=0\}}{u^{\gamma}dx},
\]
with $0<s,\gamma<1$, subject to $u \geq 0$. The first part of the energy is related to the extension of the fractional Laplacian operator, and the second one is considered as a penalty for the function $u$ being greater than 0. The set $\{u=0\}$ only lies on $\{y=0\}$, and is non-trivial if $u$ is small enough on $\partial B_1^{n+1}\cap \{y>0\}$. The boundary of the set $\{u>0\}$ in the topology of $\R^{n}$ is called the free boundary. And there is one important number $\beta=\frac{2s}{2-\gamma}$, which is the critical exponent in the scaling of the energy.\\

This problem is a non-local analogue of the problem introduced in \cite{AP} by Alt and Philips, in which a free boundary problem of the energy functional $\int_{B_1^n(0)}{|\nabla u|^2+|max(u,0)|^{\gamma}}$ is discussed. We are now considering the case for the fractional Laplacian operator instead of Laplacian, and this is an intersection of one-phase free boundary problems and non-local integrodifferential operators. Heuristically, two limiting classical problems, one as $\gamma \to 0$ is the Bornuelli type one-phase free boundary problem from the minimization of $J_0(u)=\frac{1}{2}\int{|\nabla u|^2+\chi_{\{u>0\}}}$, discussed by Caffarelli and Salsa in \cite{C1}; and the other one as $\gamma \to 1$ is the obstacle problem from the minimimzation of $J_1(u)=\frac{1}{2}\int{|\nabla u|^2+max(u,0)}$, discussed by Caffarelli in \cite{C2}. Analogues of both problems in the fractional cases are also dicussed in \cite{C3}\cite{D1}\cite{D2}\cite{D3} for the Bornuelli type problems, and in \cite{C4}\cite{C6} for the thin obstacle problems. These are the inspirations for our minimization problem, which is an intermediate case of the fractional one-phase cavitation problem and obstacle problem. \\

There are some previous results on the properties of the minimizers of the energy $J_{\gamma}(u)$. In \cite{Y1} by Ray Yang, optimal regularity is proved, that the minimizer is $C^{\beta}$ continuous if $\beta<1$ and is $C^{\alpha}$ continuous for any $\alpha<1$, if $\beta\geq 1$. And the minimizer along the set $\{y=0\}$ is $C^{\beta}$ continuous if $\beta<1$ and is $C^{1,\beta-1}$ continuous if $\beta\geq 1$. Non-degeneracy of the minimizer is also proved, that $\sup_{x\in B_r^n(x_0)}{u(x,0)}\geq C(n,s,\gamma)r^{\beta}$ if $x_0$ is a free boundary point. \\

This paper is divided into two parts. In the first part, we use Weiss type monotonicity formula introduced in \cite{W1} to prove blow-up profiles are homogeneous of degree $\beta=\frac{2s}{2-\gamma}$, the critical exponent, and the blow-up limit is unique regardless of subsequences, using Monneau type monotonicty formula introduced in \cite{M1}. We also prove that the half-plane solution is unique up to rotation. The other part is to prove there exists a small constant $\gamma_0>0$, such that for each $0<\gamma<\gamma_0$, flatness condition of the free boundary implies $C^{1,\theta}$ regualrity, applying the method introduced in \cite{D1} by De Silva, Savin and Sire. \\

We define the scaling of the minimizer near a free boundary point $(x_0,0)$,
\[
u_R(x,y)=\frac{u(R(x-x_0)+x_0,Ry)}{R^{\beta}}, 
\]
and the blow-up of the minimizer at a point $(x_0,0)$ on the free boundary is the limit of $u_R$ as $R \to 0$.\\

The fractional Laplacian is a non-local integral operator defined as
\[
(-\Delta)^su(x)=C_{n,s}P.V.\int_{\R^n}{\frac{u(x)-u(y)}{|x-y|^{n+2s}}dy},
\]
\[
C_{n,s}=\frac{4^s\Gamma(n/2+s)}{\pi^{n/2}|\Gamma(-s)|},
\]
with a corresponding nonlocal energy
\[
E(u)=\int_{\R^n \times \R^n}{\frac{|u(x)-u(y)|^2}{|x-y|^{n+2s}}dydx}
\]
which is hard to handle. So an extension of the function to one extra dimension is introduced by Caffarelli and Silvestre in \cite{C5}, transforming a non-local equation on $\R^n$ to an elliptic equation on the upper half space $\R^{n}\times \R^+$ with a Neumann boundary condition. Consider a fractional Laplacian equation $(-\Delta)^su(x)=f(x)$ in $\R^n$, and $u \in H^s(\R^n)$. Define an extension $U(x,y)$ in $\R^{n}\times \R^+$ by a Poisson kernel in Section 2.4 in \cite{C5}, such that $U(x,0)=u(x)$ and the extension $U(x,y)$ satisfies the following equations with Neumann boundary condition,
\begin{equation}\label{eq1}
div (y^{1-2s}\nabla U(x,y))=0 \ \ \text{in} \ \  \R^n\times \R^+
\end{equation}
and
\begin{equation}\label{eq2}
\lim_{y \to 0^+}{y^{1-2s}\partial_yU(x,y)}=-C_{n,s}(-\Delta)^su(x) \ \ \text{in} \ \  \R^n.
\end{equation}
And there is a natural energy 
\[
E(U)=\int_{\R^n\times \R^+}{y^{1-2s}|\nabla U(x,y)|^2dxdy}
\]
corresponding to the Euler-Lagrange equation \eqref{eq1}.\\

From the Euler-Lagrange equation of the energy 
\[
J_{\gamma}(u)=\frac{1}{2}\int_{(B_1^{n+1})^+}{y^{1-2s}|\nabla u(x,y)|^2dxdy}+\int_{B_1^{n}\times \{y=0\}}{u^{\gamma}dx},
\] 
the minimizer satisfies a second order PDE,
\[
div(y^{\alpha}\nabla u)=0
\]
in the upper half ball $(B_1^{n+1})^+$ in a distributional sense, and 
\[
\lim_{y\to 0^+}{y^{\alpha}\partial_yu(x,y)}=\gamma u^{\gamma-1}(x,0)
\]
on $\{u>0\}\cap \{y=0\}$. In the paper we denote $\alpha=1-2s$.

\section{Preliminaries}
Throughout this paper, we have the following notations. A point in the upper half space is $X=(x,y) \in (\R^{n+1})^+=\R^n\times \R^+$; the upper half ball of radius $R$ centered at $0$ is $(B_R^{n+1})^+=\{(x,y)\in (\R^{n+1})^+, |(x,y)|<R, y>0\}$, its boundary in $\{y>0\}$ is $(\partial B_R^{n+1})^+=\{(x,y)\in (\R^{n+1})^+, |(x,y)|=R, y>0\}$, and its boundary in $\{y=0\}$ is $B_R^{n}=\{(x,y)\in (\R^{n+1})^+, |x|<R, y=0\}$. Sometimes, we denote $B_1^+$ as $(B_1^{n+1})^+$ for simplification.\\
We define $\alpha=1-2s$ with $s\in(0,1)$ the order of fractional Laplacian, and $\beta=\frac{2s}{2-\gamma}$ is the critical scaling exponent with $0<\gamma<1$.\\
We denote the energy $J(u)=J_{\gamma}(u)$ by
\[
J_{\gamma}(u)=\frac{1}{2}\int_{(B_1^{n+1})^+}{y^{1-2s}|\nabla u(x,y)|^2dxdy}+\int_{B_1^{n}\times \{y=0\}}{u^{\gamma}dx}.
\]
The set $\{u=0\}$ which necessarily lies on $\{y=0\}$ is called the contact set of $u$, and we denote the free boundary $F(u)$ as the interface between the set $\{u>0\} \cap \{y=0\}$ and the contact set.

\subsection{Scaling of the problem}
Define $u_{\lambda}(X)=\lambda^{-\beta}u(\lambda X)$, $X=(x,y)\in (B_{\lambda^{-1}}^{n+1})^+$, then by the change of variables,
\begin{equation*}
\begin{aligned}
J(B_{\lambda^{-1}},u_{\lambda})&=\frac{1}{2}\int_{(B_{\lambda^{-1}}^{n+1})^+}{y^{\alpha}\lambda^{-2\beta}|\nabla u(\lambda x,\lambda y)|^2dxdy}\\
&+\int_{B_{\lambda^{-1}}^{n}\times \{y=0\}}{\lambda^{-\beta \gamma}u^{\gamma}(\lambda x)dx}\\
&=\frac{1}{2}\lambda^{-n+2-2\beta-\alpha}\int_{(B_1^{n+1})^+}{y^{\alpha}|\nabla u(x,y)|^2dxdy}\\
&+\lambda^{-n+1-\beta \gamma}\int_{B_1^{n}\times \{y=0\}}{u^{\gamma}dx}.
\end{aligned}
\end{equation*}
We require two equal exponents of $\lambda$, and this leads to
\[
\beta=\frac{2s}{2-\gamma}=\frac{1-\alpha}{2-\gamma},
\]
and thus
\[
J((B_{\lambda^{-1}}^{n+1})^+,u_{\lambda})=\lambda^{-n+1-\beta \gamma}J((B_1^{n+1})^+,u).
\]
So if $u$ is a minimizer for the energy in $(B_1^{n+1})^+$, then $u_{\lambda}$ is a minimizer in $(B_{\lambda^{-1}}^{n+1})^+$.

\subsection{Function space}
We are considering minimizers of energy 
\[
J_{\gamma}(u)=\frac{1}{2}\int_{(B_1^{n+1})^+}{y^{\alpha}|\nabla u(x,y)|^2dxdy}+\int_{B_1^{n}\times \{y=0\}}{u^{\gamma}dx}
\]
in the space $H^1(y^{\alpha},B_1^+)$, which is a weighted $H^1$ space, with norm
\[
\|u\|_{H^1(y^{\alpha},B_1^+)}=(\int_{(B_1^{n+1})^+}{y^{\alpha}(|\nabla u|^2+u^2)dxdy})^{1/2},
\]
and seminorm
\[
[u]_{H^1(y^{\alpha},B_1^+)}=(\int_{(B_1^{n+1})^+}{y^{\alpha}|\nabla u|^2dxdy})^{1/2}.
\]
From the extension theorem of Caffarelli and Silverstre in \cite{C5}, trace of any $H^1(y^{\alpha},B_1^+)$ function lies in $H^s(B_1^n(0))$, and Sobolev embedding makes sure it also lies in $L^2(B_1^n(0))$.

\section{Blow-ups are homogeneous of degree $\beta$}
In this section, we will use Weiss type monotonicity formula to prove the blow-up of the energy minimizer at every free boundary point is homogeneous of degree $\beta$.\\
If u is a minimizer of the energy $J(u)$, then it satisfies
\begin{equation}\label{eq3}
\begin{aligned}
div(y^{\alpha}\nabla u)=0 \ \  \text{in}\ \ (B_1^{n+1})^+,
\end{aligned}
\end{equation}
and
\begin{equation}\label{eq4}
\begin{aligned}
\lim_{y \to 0}{y^{\alpha}\partial_yu(x,y)}= \gamma u^{\gamma-1}(x,0) \ \  \text{on} \ \ B_1^n\cap \{u>0\}.
\end{aligned}
\end{equation}

Here we introduce a boundary adjusted energy and define 
\begin{equation}\label{Weiss}
\begin{aligned}
W(R,u)&=R^{-(n-1+\frac{2-\alpha \gamma}{2-\gamma})}\int_{(B_R^{n+1})^+}{y^{\alpha}|\nabla u|^2dxdy}\\
&+2 R^{-(n-1+\frac{2-\alpha \gamma}{2-\gamma})}\int_{B_R^n}{u^{\gamma}dx}\\
&-\beta R^{-(n+\frac{2-\alpha \gamma}{2-\gamma})}\int_{(\partial B_R^{n+1})^+}{y^{\alpha}u^2d\sigma}.
\end{aligned}
\end{equation}
This energy is invariant under scaling, 
\begin{equation}\label{Weiss_homo}
W(R\rho,u)=W(\rho,u_R),
\end{equation}
where
\[
u_R(x,y)=\frac{u(Rx,Ry)}{R^{\beta}}.
\]

\begin{theorem}[Weiss type monotonicity formula]\label{Weissmono}
\label{Weiss_monotonicity}
If $u$ is a minimizer of $J(u)$ and $0$ is a free boundary point, then the boundary adjusted energy $W(R,u)$ satisfies the monotonicity formula
\begin{equation}\label{Weiss_main}
\frac{d}{dR}W(R,u)=2R^{-(n-1+\frac{2-\alpha \gamma}{2-\gamma})}\int_{(\partial B_R^{n+1})^+}{y^{\alpha}(u_{\nu}-\beta \frac{u}{R})^2d\sigma}.
\end{equation}
Moreover, when $\frac{d}{dR}W(R,u)=0$, it is equivalent that
\[
0=<(x,y),\nabla u(x,y)>-\beta u(x,y)=\frac{d}{d\rho}|_{\rho=1}\frac{u(\rho x,\rho y)}{\rho^{\beta}}
\]
a.e. on $(\partial B_1^{n+1})^+$, which means u is homogeneous of degree $\beta$.
\end{theorem}

\begin{proof}[Proof of Weiss type monotonicity formula]
If u is a minimizer of the energy $J(u)$, then it satisfies
$div(y^{\alpha}\nabla u)=0$, $div(y^{\alpha}u\nabla u)=y^{\alpha}|\nabla u|^2$ in $(B_1^{n+1})^+$and
$\lim_{y \to 0}{y^{\alpha}\partial_yu(x,y)}= \gamma u^{\gamma-1}(x,0)$ on $B_1^n\cap \{u>0\}$. And the following equalities are obtained:

\begin{equation}\label{Weiss1}
\begin{aligned}
\int_{(B_R^{n+1})^+}{y^{\alpha}|\nabla u|^2dxdy}=\int_{(\partial B_R^{n+1})^+}{y^{\alpha}u u_{\nu}d\sigma}- \gamma \int_{B_R^n}{u^{\gamma}dx};
\end{aligned}
\end{equation}

\begin{equation}\label{Weiss2}
\begin{aligned}
\int_{B_R^n}{<x,\nabla u^{\gamma}>dx}=R\int_{\partial B_R^n}{u^{\gamma}d\sigma}-n\int_{B_R^n}{u^{\gamma}dx};
\end{aligned}
\end{equation}

\begin{equation}\label{Weiss3}
\begin{aligned}
&\frac{d}{dR}(\int_{(\partial B_R^{n+1})^+}{y^{\alpha}u^2d\sigma})\\
&=\frac{n+\alpha}{R}\int_{(\partial B_R^{n+1})^+}{y^{\alpha}u^2d\sigma}+2\int_{(\partial B_R^{n+1})^+}{y^{\alpha}uu_{\nu}d\sigma};
\end{aligned}
\end{equation}

\begin{equation}\label{Weiss4}
\begin{aligned}
&(n+\alpha-1)\int_{(B_R^{n+1})^+}{y^{\alpha}|\nabla u|^2dxdy}\\
&=R\int_{(\partial B_R^{n+1})^+}{y^{\alpha}(|\nabla u|^2-2u_{\nu}^2)d\sigma}+2 \int_{B_R^n}{<x,\nabla u^{\gamma}>dx}\\ 
&=R\int_{(\partial B_R^{n+1})^+}{y^{\alpha}(|\nabla u|^2-2u_{\nu}^2)d\sigma}+2 R \int_{\partial B_R^n}{u^{\gamma}d\sigma}-2 n\int_{B_R^n}{u^{\gamma}dx}.
\end{aligned}
\end{equation}

Calculate derivative of $W(R,u)$ with respect to $R$ and we can get:
\begin{equation*}
\begin{aligned}
R^{n+\frac{2-\alpha \gamma}{2-\gamma}}\frac{d}{dR}W(R,u)&=-(n-1+\frac{2-\alpha \gamma}{2-\gamma})\int_{(B_R^{n+1})^+}{y^{\alpha}|\nabla u|^2dxdy}\ \  (I_1)\\
&+R\int_{(\partial B_R^{n+1})^+}{y^{\alpha}|\nabla u|^2d\sigma} \ \ (I_2)\\
&-2 (n+\frac{\gamma-\alpha \gamma}{2-\gamma})\int_{B_R^n}{u^{\gamma}dx}\ \  (I_3)\\
&+2 R\int_{\partial B_R^n}{u^{\gamma}d\sigma}  \ \ (I_4)\\
&+\beta (n+\frac{2-\alpha \gamma}{2-\gamma}) R^{-1}\int_{(\partial B_R^{n+1})^+}{y^{\alpha}u^2d\sigma} \ \ (I_5)\\
&-\beta \frac{n+\alpha}{R}\int_{(\partial B_R^{n+1})^+}{y^{\alpha}u^2d\sigma} \ \ (I_6)\\
&-2\beta \int_{(\partial B_R^{n+1})^+}{y^{\alpha}uu_{\nu}d\sigma}. \ \ (I_7)\\
\end{aligned}
\end{equation*}
Apply \eqref{Weiss4} and \eqref{Weiss1}, then
\begin{equation*}
\begin{aligned}
R^{n+\frac{2-\alpha \gamma}{2-\gamma}}[(I_1)+(I_2)]&=2R\int_{(\partial B_R^{n+1})^+}{y^{\alpha}u_{\nu}^2d\sigma}\\
&-2 R\int_{\partial B_R^{n}}{u^{\gamma}}+2n \int_{B_R^{n}}{u^{\gamma}}\\
&-\frac{2-2\alpha}{2-\gamma} (\int_{(\partial B_R^{n+1})^+}{y^{\alpha}u u_{\nu}d\sigma}- \gamma \int_{B_R^n}{u^{\gamma}dx}).
\end{aligned}
\end{equation*}
After adding $(I_3)$ and $(I_4)$ we obtain:
\begin{equation*}
\begin{aligned}
(I_1)+(I_2)+(I_3)+(I_4)&=2R^{-(n-1+\frac{2-\alpha \gamma}{2-\gamma})}\int_{(\partial B_R^{n+1})^+}{y^{\alpha}u_{\nu}^2d\sigma}\\
&-\frac{2-2\alpha}{2-\gamma}R^{-(n+\frac{2-\alpha \gamma}{2-\gamma})}\int_{(\partial B_R^{n+1})^+}{y^{\alpha}u^2d\sigma},
\end{aligned}
\end{equation*}
And adding the last three terms $(I_5)$, $(I_6)$ and $(I_7)$, we can calculate
\[
\frac{d}{dR}W(R,u)=2R^{-(n-1+\frac{2-\alpha \gamma}{2-\gamma})}\int_{(\partial B_R^{n+1})^+}{y^{\alpha}(u_{\nu}-\beta \frac{u}{R})^2d\sigma}.
\]
\end{proof}

Let $0\in \partial\{u>0\}\cap \{y=0\}$, and consider the function $u_r(X)=r^{-\beta}u(rX)$. As $r_k \to 0$, $u_{r_k}$ converges to $u_0$ weakly in $H^1(y^{\alpha}, (\R^{n+1})^+)$. Pass to a subsequence (still denoted by $r_k$), $u_{r_k}\to u_0$ in $L^2_{loc}(y^{\alpha}, (\R^{n+1})^+)$, and in $L^2_{loc}(\R^n \times\{y=0\})$. And the blow-up $u_0$ is a global minimizer of $J(\Omega, u)$ on any $\Omega \subset (\R^{n+1})^+$. Thus, $W(r_k,u)$ is a bounded non-decreasing function of $r_k$ by Theorem \ref{Weissmono}, if $u$ is a minimizer. Then with the boundedness of the sequence $\{u_{r_k}\}$ in $H^1(y^{\alpha}, (\R^{n+1})^+)$, we can prove the following corollary.

\begin{corollary}[Blow-ups are homogeneous of degree $\beta$] If $u$ is a minimizer of $J(u)$, then the blow-up limit $u_0$ at every free boundary point is homogeneous of degree $\beta$.
\end{corollary}

\begin{proof}
Since $W(\rho r,u)=W(\rho, u_r)$ by the scaling property of $W$, then for any $R>0$,
\[
W(R,u_0)=\lim_{k \to \infty}{W(R,u_{r_k})}=\lim_{k \to \infty}{W(Rr_k,u)}=W(0^+,u)
\]
is a constant, since $W(Rr_k,u)$ is a bounded non-decreasing function of $r_k$ by Theorem \ref{Weissmono}. Thus
\[
\frac{d}{dR}W(R,u_0)=0,
\]
and this implies that $u_0$ is homogeneous of degree $\beta$.
\end{proof}

\section{Uniqueness of the blow-up profile regardless of subsequences}
We define
\[
u_0(x,y)=\lim_{r_k\to 0}{\frac{u(r_kx,r_ky)}{r_k^{\beta}}}
\]
as the blow-up profile of the minimizer $u$ near a free boundary point 0. But different subsequences may lead to different blow-up profile $u_0$. In this section, our aim is to prove that the limit is unique regardless of subsequences. \\

The blow-up $u_0$ satisfies the same equations \eqref{eq3} and \eqref{eq4} as $u$ does. For any function $p\geq 0$ homogeneous of degree $\beta$ and satisfying these equations \eqref{eq3} and \eqref{eq4},
\begin{equation*}
\begin{aligned}
W(R,p)&= R^{-(n+\frac{2-\alpha \gamma}{2-\gamma}-1)}\int_{(\partial B_R^{n+1})^+}{y^{\alpha}p(p_{\nu}-\beta\frac{p}{R})d\sigma}\\
&+(2-\gamma)  R^{-(n+\frac{2-\alpha \gamma}{2-\gamma}-1)}\int_{B_R^n}{p^{\gamma}d\sigma}\\
&=(2-\gamma)  R^{-(n+\frac{2-\alpha \gamma}{2-\gamma}-1)}\int_{B_R^n}{p^{\gamma}dx}
\end{aligned}
\end{equation*}

Now we prove the Monneau type monotonicty formula. Let $w=u-p$, and define
\[
M(R,u,p)=R^{-(n+\frac{2-\alpha \gamma}{2-\gamma})}\int_{(\partial B_R^{n+1})^+}{(u-p)^{2}d\sigma}=R^{-(n+\frac{2-\alpha \gamma}{2-\gamma})}\int_{(\partial B_R^{n+1})^+}{w^{2}d\sigma}.
\]
Here $p$ satisfies the same equations as $u$ does.\\ Then

\begin{equation*}
\begin{aligned}
\frac{d}{dR}M(R,u,p)&= \frac{d}{dR}\int_{(\partial B_1^{n+1})^+}{y^{\alpha}\frac{w^2(Rx,Ry)}{R^{2\beta}}d\sigma}\\
&=\int_{(\partial B_1^{n+1})^+}{y^{\alpha}\frac{2w(RX)(RX\cdot \nabla w(RX)-\beta w(RX))}{R^{2\beta+1}}d\sigma}\\
&=2R^{-(n+\frac{2-\alpha \gamma}{2-\gamma})}\int_{(\partial B_R^{n+1})^+}{y^{\alpha}w(x,y)(w_{\nu}-\beta \frac{w}{R})d\sigma}.
\end{aligned}
\end{equation*}

We have the following equality
\begin{equation*}
\begin{aligned}
W(R,u)&=W(R,u)-W(R,p)+(2-\gamma)  R^{-(n+\frac{2-\alpha \gamma}{2-\gamma}-1)}\int_{B_R^n}{p^{\gamma}dx}\\
&= R^{-(n+\frac{2-\alpha \gamma}{2-\gamma}-1)}\int_{(B_R^{n+1})^+}{y^{\alpha}(|\nabla u|^2-|\nabla p|^2)dxdy}\\
&-\beta  R^{-(n+\frac{2-\alpha \gamma}{2-\gamma})}\int_{(\partial B_R^{n+1})^+}{y^{\alpha}(u^2-p^2)d\sigma}\\
&+2  R^{-(n+\frac{2-\alpha \gamma}{2-\gamma}-1)}\int_{B_R^n}{(u^{\gamma}-p^{\gamma})dx}\\
&+(2-\gamma)  R^{-(n+\frac{2-\alpha \gamma}{2-\gamma}-1)}\int_{B_R^n}{p^{\gamma}dx}\\
&= R^{-(n+\frac{2-\alpha \gamma}{2-\gamma}-1)}\int_{(B_R^{n+1})^+}{y^{\alpha}|\nabla w|^2dxdy}-\beta  R^{-(n+\frac{2-\alpha \gamma}{2-\gamma})}\int_{(\partial B_R^{n+1})^+}{y^{\alpha}w^2d\sigma}\\
&+2R^{-(n+\frac{2-\alpha \gamma}{2-\gamma}-1)}\int_{(B_R^{n+1})^+}{y^{\alpha}\nabla w \cdot \nabla p dxdy}-2\beta  R^{-(n+\frac{2-\alpha \gamma}{2-\gamma})}\int_{(\partial B_R^{n+1})^+}{y^{\alpha}wpd\sigma}\\
&+2  R^{-(n+\frac{2-\alpha \gamma}{2-\gamma}-1)}\int_{B_R^n}{u^{\gamma}dx}-\gamma  R^{-(n+\frac{2-\alpha \gamma}{2-\gamma}-1)}\int_{B_R^n}{p^{\gamma}dx}\\
\end{aligned}
\end{equation*}
And since
\[
div(y^{\alpha}w \nabla p)=y^{\alpha}\nabla w \cdot \nabla p+wdiv(y^{\alpha}\nabla p)=y^{\alpha}\nabla w \cdot \nabla p,
\]
we can see
\[
\int_{(B_R^{n+1})^+}{y^{\alpha}\nabla w \cdot \nabla p dxdy}=\int_{(\partial B_R^{n+1})^+}{y^{\alpha}wp_{\nu} dxdy}-\int_{B_R^n}{w  \gamma p^{\gamma-1}dx}.
\]
Thus plug in this equation and since $p$ is homogeneous of degree $\beta$, we are able to obtain
\begin{equation*}
\begin{aligned}
W(R,u)&= R^{-(n+\frac{2-\alpha \gamma}{2-\gamma}-1)}\int_{(B_R^{n+1})^+}{y^{\alpha}|\nabla w|^2dxdy}-\beta  R^{-(n+\frac{2-\alpha \gamma}{2-\gamma})}\int_{(\partial B_R^{n+1})^+}{y^{\alpha}w^2d\sigma}\\
&+2R^{-(n+\frac{2-\alpha \gamma}{2-\gamma}-1)}\int_{(\partial B_R^{n+1})^+}{y^{\alpha}wp_{\nu} dxdy}-2\beta  R^{-(n+\frac{2-\alpha \gamma}{2-\gamma})}\int_{(\partial B_R^{n+1})^+}{y^{\alpha}wpd\sigma}\\
&+  R^{-(n+\frac{2-\alpha \gamma}{2-\gamma}-1)}\int_{B_R^n}{(2u^{\gamma}-\gamma p^{\gamma}- 2\gamma wp^{\gamma-1}) dx}\\
&=R^{-(n+\frac{2-\alpha \gamma}{2-\gamma}-1)}\int_{(B_R^{n+1})^+}{y^{\alpha}|\nabla w|^2dxdy}-\beta  R^{-(n+\frac{2-\alpha \gamma}{2-\gamma})}\int_{(\partial B_R^{n+1})^+}{y^{\alpha}w^2d\sigma}\\
&+  R^{-(n+\frac{2-\alpha \gamma}{2-\gamma}-1)}\int_{B_R^n}{(2u^{\gamma}-\gamma p^{\gamma}- 2\gamma wp^{\gamma-1}) dx}
\end{aligned}
\end{equation*}

Also, we can see
\[
\int_{(B_R^{n+1})^+}{w div(y^{\alpha}\nabla w) dxdy}=\int_{(\partial B_R^{n+1})^+}{y^{\alpha}ww_{\nu} dxdy}-\int_{B_R^n}{y^{\alpha}|\nabla w|^2dx},
\]
and
\[
div(y^{\alpha}\nabla w)=div(y^{\alpha}\nabla u)-div(y^{\alpha}\nabla p)=0.
\]
Then we will obtain
\begin{equation*}
\begin{aligned}
W(R,u)&=R^{-(n+\frac{2-\alpha \gamma}{2-\gamma}-1)}\int_{(B_R^{n+1})^+}{y^{\alpha}|\nabla w|^2dxdy}-\beta  R^{-(n+\frac{2-\alpha \gamma}{2-\gamma})}\int_{(\partial B_R^{n+1})^+}{y^{\alpha}w^2d\sigma}\\
&+  R^{-(n+\frac{2-\alpha \gamma}{2-\gamma}-1)}\int_{B_R^n}{(2u^{\gamma}-\gamma p^{\gamma}- 2\gamma wp^{\gamma-1}) dx}\\
&=R\frac{d}{dR}M(R,u,p)+  R^{-(n+\frac{2-\alpha \gamma}{2-\gamma}-1)}\int_{B_R^n}{(2u^{\gamma}-\gamma p^{\gamma}- 2\gamma wp^{\gamma-1}) dx}.
\end{aligned}
\end{equation*}
Since $0<\gamma<1$, so function $f(x)=x^{\gamma}$ is concave on $\R^+$, and thus
\[
u^{\gamma}=(w+p)^{\gamma}\leq p^{\gamma}+\gamma p^{\gamma-1}w,
\]
since $u,p\geq 0$.
Therefore,
\begin{equation*}
\begin{aligned}
W(R,u)&\leq R\frac{d}{dR}M(R,u,p)+(2-\gamma)  R^{-(n+\frac{2-\alpha \gamma}{2-\gamma}-1)}\int_{B_R^n}{p^{\gamma}dx}.
\end{aligned}
\end{equation*}

We know there is a subsequence $u_{r_j}$ such that
\[
M(0^+,u,u_0)=\lim_{r_j\to 0}M(1,u_{r_j},u_0)=0,
\]
and we also know 
\begin{equation*}
\begin{aligned}
R\frac{d}{dR}M(R,u,u_0)&\geq W(R,u)-(2-\gamma)  R^{-(n+\frac{2-\alpha \gamma}{2-\gamma}-1)}\int_{B_R^n}{u_0^{\gamma}dx}\\
&\geq W(0^+,u)-(2-\gamma)  R^{-(n+\frac{2-\alpha \gamma}{2-\gamma}-1)}\int_{B_R^n}{u_0^{\gamma}dx}\\
&= W(R,u_0)-(2-\gamma)  R^{-(n+\frac{2-\alpha \gamma}{2-\gamma}-1)}\int_{B_R^n}{u_0^{\gamma}dx}\\
&=\gamma  R^{-(n+\frac{2-\alpha \gamma}{2-\gamma}-1)}\int_{B_R^n}{u_0^{\gamma}dx}\\
&\geq 0.
\end{aligned}
\end{equation*}
Therefore, 
\[
\lim_{R \to 0}{M(R,u,u_0)}=\lim_{R \to 0}{M(1,u_R,u_0)}=0,
\]
which means the blow-up profile is unique regardless of subsequence.

\section{Uniqueness of half-plane solution}
In this section, we apply the method introduced in \cite{C3} to prove the following theorem.

\begin{theorem} If u is the minimizer in $(\mathbb{R}^{n+1})^+$, and $u(x,0)=A(x_n)_+^{\beta}$, then 
\[
A=(\frac{\beta-s}{-\beta A_1})^{1/(2-\gamma)}
\]
is determined by $s$ and $\gamma$, where
\[
A_1=-\frac{C_{1,s}}{2}\int_{-\infty}^{\infty}{\frac{(1+y)_+^{\beta}+(1-y)_+^{\beta}-2}{|y|^{1+2s}}dy}<0,
\]
with constant
\[
C_{1,s}=\frac{4^s\Gamma(1/2+s)}{\pi^{1/2}|\Gamma(-s)|}.
\]
\end{theorem}
\begin{proof}
First we prove the theorem when $n=1$. Let
\[
J(u)=\int_{B_1^+}y^{\alpha}|\nabla u|^2dxdy+\int_{-1}^{1}{u^{\gamma}dx},
\]
and consider $U_0(x,y)$ as the extension of $u_0(x)=(x)_{+}^{\beta}$. Define
\[
u_{\epsilon}(x)=\frac{(x+\epsilon)_+^{\beta}}{(1+\epsilon)^{\beta}},
\]
and
\[
\tilde{u}_{\epsilon}=\begin{cases} 
      u_{\epsilon}(x) & |x|\leq 1 \\
      u_0(x) & |x|>1.
   \end{cases}
\]
And the function $U_{\epsilon}(x,y)$ satisfies the following equation:
\[
\begin{cases} 
     div(y^{\alpha}\nabla U_{\epsilon}(x,y))=0 \ \ in \ \ (B_1^2)^+ \\
     U_{\epsilon}(x,0)=u_{\epsilon}(x) \ \ \ \ \ \ |x|\leq 1\\
     U_{\epsilon}(x,y)=U_0(x,y) \ \ on \ \ (\partial B_1^2)^+.
   \end{cases}
\]

If $AU_0$ is a local minimizer of $J(u)$, then $J(AU_0)\leq J(AU_\epsilon)$ for any $\epsilon$, that is
\[
A^2\int_{B_1^+}y^{\alpha}|\nabla U_0|^2dxdy+A^{\gamma} \int_{-1}^{1}{u_0^{\gamma}dx} \leq A^2\int_{B_1^+}y^{\alpha}|\nabla U_{\epsilon}|^2dxdy+A^{\gamma}\int_{-1}^{1}{u_{\epsilon}^{\gamma}dx}.
\]
We can see
\begin{equation*}
\begin{aligned}
\int_{-1}^{1}{u_{\epsilon}^{\gamma}dx}-\int_{-1}^{1}{u_0^{\gamma}dx}&=\frac{1}{(1+\epsilon)^{\beta\gamma}}\frac{1}{1+\beta\gamma}(1+\epsilon)^{\beta\gamma+1}-\frac{1}{1+\beta\gamma}\\
&=\frac{\epsilon}{1+\beta\gamma},
\end{aligned}
\end{equation*}
and
\begin{equation*}
\begin{aligned}
&(-1)[\int{y^{\alpha}|\nabla U_0|^2}-\int{y^{\alpha}|\nabla U_{\epsilon}|^2}]\\
&=\int{y^{\alpha}|\nabla (U_0-U_{\epsilon})|^2}+2\int{y^{\alpha}\nabla U_0 \nabla(U_\epsilon-U_0)}\\
&=I_2+2I_1.
\end{aligned}
\end{equation*}
First let us calculate $I_1$:
\begin{equation*}
\begin{aligned}
I_1&=\int{y^{\alpha}\nabla U_0 \nabla(U_\epsilon-U_0)}\\
&=\int_{(B_1^2)^+}{div(y^{\alpha}\nabla U_0(U_\epsilon-U_0))}-\int{((U_\epsilon-U_0)div(y^{\alpha}\nabla U_0)}\\
&=\int_{(\partial B_1^2)^+}{y^{\alpha}(U_0)_{\nu}(U_\epsilon-U_0))}-\int_{-1}^1{(\lim_{y\to 0^+}{y^{\alpha}\partial_yU_0})(U_\epsilon-U_0)}\\
&=\int_{-1}^1{(-\Delta)^su_0(x)(u_\epsilon-u_0)}.\\
\end{aligned}
\end{equation*}
By the homogeneity property of $u_0$, we can calculate that when $x>0$,
\begin{equation*}
\begin{aligned}
(-\Delta)^su_0(x)&=-\frac{C_{1,s}}{2}\int_{-\infty}^{\infty}{\frac{(x+y)_+^{\beta}+(x-y)_+^{\beta}-2(x)_+^{\beta}}{|y|^{1+2s}}dy}\\
&=x^{\beta-2s}\frac{-C_{1,s}}{2}\int_{-\infty}^{\infty}{\frac{(1+y)_+^{\beta}+(1-y)_+^{\beta}-2}{|y|^{1+2s}}dy}\\
&=A_1x^{\beta-2s},
\end{aligned}
\end{equation*}
and when $x<0$,
\begin{equation*}
\begin{aligned}
(-\Delta)^su_0(x)&=-(-x)^{\beta-2s}C_{1,s}P.V.\int_{1}^{\infty}{\frac{(y-1)^{\beta}}{|y|^{1+2s}}dy}\\
&=A_2(-x)^{\beta-2s}.
\end{aligned}
\end{equation*}
Notice that $A_1,A_2<0$, with
\[
A_1=-\frac{C_{1,s}}{2}\int_{-\infty}^{\infty}{\frac{(1+y)_+^{\beta}+(1-y)_+^{\beta}-2}{|y|^{1+2s}}dy},
\]
and
\[
A_2=-C_{1,s}P.V.\int_{1}^{\infty}{\frac{(y-1)^{\beta}}{|y|^{1+2s}}dy}.
\]
Then we can calculate that
\begin{equation*}
\begin{aligned}
I_1&=\int_{-1}^1{(-\Delta)^su_0(x)(u_\epsilon-u_0)}\\
&=A_1\int_0^1{x^{\beta-2s}(\frac{(x+\epsilon)_+^{\beta}}{(1+\epsilon)^{\beta}}-x^{\beta})dx}\\
&+A_2\int_{-1}^0{(-x)^{\beta-2s}\frac{(x+\epsilon)_+^{\beta}}{(1+\epsilon)^{\beta}}dx}\\
&=A_1\beta(\frac{1}{2\beta-2s}-\frac{1}{2\beta-2s+1})\epsilon+o(\epsilon).
\end{aligned}
\end{equation*}
Then we try to calculate $I_2$, 
\begin{equation*}
\begin{aligned}
I_2&=\int{y^{\alpha}|\nabla (U_0-U_{\epsilon})|^2}\\
&=\int{y^{\alpha} (U_{\epsilon}-U_0)\nabla(U_{\epsilon}-U_0)}-\int{(U_{\epsilon}-U_0)div(y^{\alpha}\nabla(U_{\epsilon}-U_0))}\\
&=\int_{(\partial B_1^2)^+}{y^{\alpha} (U_{\epsilon}-U_0)(U_{\epsilon}-U_0)_{\nu}}-\int_{-1}^{1}{(\lim_{y\to 0^+}{y^{\alpha}\partial_y(U_{\epsilon}-U_0}))(U_{\epsilon}-U_0)}\\
&=\int_{-1}^{1}{(-\Delta)^s(\tilde{u}_{\epsilon}-u_0)(u_{\epsilon}-u_0)}\\
&=\int_{-1}^{1}{(-\Delta)^s(\tilde{u}_\epsilon-u_{\epsilon})(u_{\epsilon}-u_0)}+\int_{-1}^{1}{(-\Delta)^s(u_{\epsilon}-u_0)(u_{\epsilon}-u_0)}.
\end{aligned}
\end{equation*}
Define 
\[
g_{\epsilon}(x)=\tilde{u}_\epsilon(x)-u_{\epsilon}(x)=\epsilon h(x)=\begin{cases} 
 0 & x\leq 1 \\
 \epsilon\beta(x^{\beta}-x^{\beta-1})+o(\epsilon)  & x>1,
   \end{cases}
\]
and
\[
(-\Delta)^s(\tilde{u}_\epsilon-u_{\epsilon})(x)=\epsilon C_{1,s}P.V.\int_{-\infty}^{\infty}{\frac{h(x+y)-h(x)}{|y|^{1+2s}}dy},
\]
and
\[
\int_{-1}^{1}{(-\Delta)^s(\tilde{u}_\epsilon-u_{\epsilon})(u_{\epsilon}-u_0)}\leq 2max|u_{\epsilon}-u_0|O(\epsilon)=o(\epsilon).
\]
Thus,
\[
I_2=o(\epsilon)+\int_{-1}^{1}{(-\Delta)^su_{\epsilon}(u_{\epsilon}-u_0)}-\int_{-1}^{1}{(-\Delta)^su_0(u_{\epsilon}-u_0)}=o(\epsilon)+I_3-I_1,
\]
where
\[
I_3=\int_{-1}^{1}{(-\Delta)^su_{\epsilon}(u_{\epsilon}-u_0)}.
\]
Since $u_{\epsilon}(x)=\frac{(x+\epsilon)_+^{\beta}}{(1+\epsilon)^{\beta}}$, then
\[
(-\Delta)^su_{\epsilon}(x)=\begin{cases} 
     \frac{1}{(1+\epsilon)^{\beta}}A_1(x+\epsilon)^{\beta-2s} & x+\epsilon>0 \\
       \frac{1}{(1+\epsilon)^{\beta}}A_2(-x-\epsilon)^{\beta-2s} & x+\epsilon>0,
   \end{cases}
\]
and we can calculate $I_3$ that
\begin{equation*}
\begin{aligned}
I_3&=\int_{-1}^{1}{(-\Delta)^su_{\epsilon}(u_{\epsilon}-u_0)}\\
&=\int_{0}^{1}{ \frac{1}{(1+\epsilon)^{\beta}}A_1(x+\epsilon)^{\beta-2s}(\frac{(x+\epsilon)^{\beta}}{(1+\epsilon)^{\beta}}-x^{\beta})dx}\\
&+\int_{-\epsilon}^{0}{ \frac{1}{(1+\epsilon)^{\beta}}A_1(x+\epsilon)^{\beta-2s}\frac{(x+\epsilon)^{\beta}}{(1+\epsilon)^{\beta}}dx}\\
&=\epsilon A_1(\frac{\beta-2s+1}{2\beta-2s+1}-\frac{\beta-2s}{2\beta-2s})+o(\epsilon)\\
&=\epsilon A_1\beta(\frac{1}{2\beta-2s}-\frac{1}{2\beta-2s+1})+o(\epsilon). 
\end{aligned}
\end{equation*}
Therefore,
\begin{equation*}
\begin{aligned}
A^2(I_2+2I_1)&=(-1)A^2[\int{y^{\alpha}|\nabla U_0|^2}-\int{y^{\alpha}|\nabla U_{\epsilon}|^2}]\\
&=-2\epsilon A^2A_1\beta(\frac{1}{2\beta-2s}-\frac{1}{2\beta-2s+1})+o(\epsilon),
\end{aligned}
\end{equation*}
and
\[
A^{\gamma}\int_{-1}^{1}{u_{\epsilon}^{\gamma}dx}-A^{\gamma}\int_{-1}^{1}{u_0^{\gamma}dx}=A^{\gamma}\frac{\epsilon}{1+\beta\gamma},
\]
and since $AU_0$ is a local minimizer of energy $J(u)$, it is required that for all $\epsilon>0$ and $\epsilon<0$,
\[
-2\epsilon A^2A_1\beta(\frac{1}{2\beta-2s}-\frac{1}{2\beta-2s+1})+o(\epsilon)\leq A^{\gamma}\frac{\epsilon}{1+\beta\gamma}.
\]
and this means that
\[
A=(\frac{\beta-s}{-\beta A_1})^{1/(2-\gamma)},
\]
and A is determined by $s$ and $\gamma$, where
\[
A_1=-\frac{C_{1,s}}{2}\int_{-\infty}^{\infty}{\frac{(1+y)_+^{\beta}+(1-y)_+^{\beta}-2}{|y|^{1+2s}}dy}<0.
\]
Just to notice, as $\gamma \to 0$, which is the case of fractional one-phase Bournelli-type problem, the constant $A_1=O(\beta-s)$ and this ensures the unique half plane minimizer will not go to 0.\\

Then applying the same proof in Theorem 1.4 in \cite{C3}, we prove the theorem for general $n$.
\end{proof}

\section{positive density when $\gamma$ is small enough}
When $\gamma \to 1$, in the thin obstacle problem \cite{C4}, near a free boundary point $x_0$, the set $\{u=0\}\cap B_1^n(x_0)$ does not always have positive density. In this section, we try to prove there exists a positive number $\gamma_0>0$, and for each $0<\gamma<\gamma_0$, the minimizer of energy $J_{\gamma}(u)$ has positive density of zero set near every free boundary point. 

\begin{theorem}
There exists $\gamma_0=\gamma_0(n,s)>0$ and $\delta>0$ such that for each $0<\gamma<\gamma_0$, if $u_{\gamma}$ is a minimizer of $J_{\gamma}(u)$, then 
\[
L^{n}(\{u_{\gamma}=0\}\cap B_1^n)\geq \delta>0.
\]

\end{theorem}
We prove the theorem by the method of compactness. And before the proof, a lemma of non-degeneracy is required.
\begin{lemma}\label{nondegeneracy}
Assume $u_{\gamma}$ is a minimizer of the energy $J_{\gamma}(u)$ and $0$ is a free boundary point. There exists a positive constant $C_0>0$ independent of $\gamma$, such that for each $x \in B_{1/2}^n \cap \{u>0\}$, 
\[
u_{\gamma}(x,0)\geq C_0(d(x,\partial\{u_{\gamma}>0\}))^{\beta}.
\]
\end{lemma}
\begin{proof}
Up to rescaling, it is enough to show, if $(x_0,0)$ is at distance 1 from the free boundary and $u_{\gamma}(x_0,0)>0$, then $\epsilon=u_{\gamma}(x_0,0)$ cannot be too small, and $\epsilon$ will not go to 0 as $\gamma \to 0$.\\

By the Harnack inequality in the upper half space (since $y^{\alpha}$ belongs to the class of $A_2$ functions defined by Muchenhoupt in \cite{M1}) and the variant boundary Harnack inequality proved in Theorem 4.1 in \cite{Y1}, there exists $c',C'>0$ independent of $\gamma$, such that
\[
0<c'\epsilon\leq u_{\gamma}(x,y)\leq C'\epsilon 
\]
in $(B_{1/2}^{n+1}(x_0,0))^+$. Take test function $\phi \in C_C^{\infty}((B_{1/2}^{n+1}(x_0,0))^+)$, and apply standard Green's indentity to obtain
\begin{equation*}
\begin{aligned}
&\int_{\{u_{\gamma}>0\}\cap B_{1/2}^n(x_0)}{u_{\gamma} (\lim_{y \to 0^+}{y^{\alpha}\partial_y\phi})-\phi(\lim_{y \to 0^+}{y^{\alpha}\partial_yu_{\gamma}})}\\
&=-\int_{(B_{1/2}^{n+1}(x_0,0))^+}{u_{\gamma} div(y^{\alpha}\nabla \phi)},
\end{aligned}
\end{equation*}
using $div(y^{\alpha}\nabla u_{\gamma})=0$ in the formula. Then
\begin{equation}\label{smallepsilon}
\begin{aligned}
|\int_{\{u_{\gamma}>0\}\cap B_{1/2}^n(x_0)}{\gamma \phi (C'\epsilon)^{\gamma-1}}|& \leq |\int_{\{u_{\gamma}>0\}\cap B_{1/2}^n(x_0)}{\gamma \phi  u_{\gamma}^{\gamma-1}}|\\
&\leq |\int_{\{u>0\}\cap B_{1/2}^n(x_0)}{u_{\gamma} (\lim_{y \to 0^+}{y^{\alpha}\partial_y\phi})}|\\
&+|\int_{(B_{1/2}^{n+1}(x_0,0))^+}{u_{\gamma} div(y^{\alpha}\nabla \phi)}|.
\end{aligned}
\end{equation}
Since $d(x,\partial\{u_{\gamma}>0\})\leq C$ if $(x,y)\in (B_{1/2}^{n+1}(x_0,0))^+$, then
\[
u(x,y)\leq \tilde{C}
\]
by $C^{\beta}$ estimates of the minimizer. And the test function $\phi \in C_C^{\infty}((B_{1}^{n+1}(x_0,0))^+)$ is smooth enough, so the integral of $\lim_{y \to 0^+}{y^{\alpha}\partial_y\phi}$ and $ div(y^{\alpha}\nabla \phi)$ are both bounded, and therefore by \eqref{smallepsilon}, $\epsilon$ cannot be too small.\\

However, $\gamma \epsilon^{\gamma-1}<\infty$ cannot ensure $\epsilon\geq C>0$ as $\gamma \to 0$. To prove that $\epsilon \geq C_0$ independent of $\gamma$, we consider a smooth function $P(x,y)\geq 0$ defined on $(B_{1/2}^{n+1}(x_0,0))^+$, with $P(x,y)=0$ in $(B_{1/4}^{n+1}(x_0,0))^+$ and $P(x,y)=2C'$ in $(B_{7/16}^{n+1}(x_0,0))^+\setminus (B_{3/8}^{n+1}(x_0,0))^+$. And define a function $v(x,y)=\min{\{u(x,y),\epsilon P(x,y)\}}$ on $(B_{1/2}^{n+1}(x_0,0))^+$. Then $J(v)\geq J(u)$ since $u(x,y)$ is the energy minimizer. First we can see
\[
\int_{(B_{1/2}^{n+1}(x_0,0))^+}{y^{\alpha}|\nabla v|^2dxdy}-\int_{(B_{1/2}^{n+1}(x_0,0))^+}{y^{\alpha}|\nabla u|^2dxdy}\leq O(\epsilon)
\]
from our definition of the function $v(x,y)$. (Same as in Section 3.4, proof of Theorem 1.2 in \cite{C3}). And
\[
\int_{B_{1/2}^n(x_0)}{v^{\gamma}-u^{\gamma}}\leq -\int_{B_{1/4}^n(x_0)}{u^{\gamma}},
\]
since $v=0$ on $B_{1/4}^n(x_0)$ and $v \leq u$ on $B_{1/2}^n(x_0)$. Therefore,
\begin{equation}\label{uvcomp}
\begin{aligned}
J(v)-J(u)\leq O(\epsilon)-\int_{B_{1/4}^n(x_0)}{u^{\gamma}}.
\end{aligned}
\end{equation}
However, $J(v)\geq J(u)$ since $u$ is the energy minimizer. Therefore, if $\epsilon \to 0$ as $\gamma \to 0$, then \eqref{uvcomp} requires $\epsilon^{\gamma} \to 0$ as $\gamma \to 0$. If not, \eqref{uvcomp} will lead to a contradiction of $u$ being the energy minimizer. Therefore, now it is required that, if $\epsilon \to 0$ as $\gamma \to 0$, then
\[
\lim_{\gamma \to 0}{\epsilon^{\gamma}}=0
\]
and
\[
\lim_{\gamma \to 0}{\gamma \epsilon^{\gamma-1}}<\infty
\]
from \eqref{uvcomp} and \eqref{smallepsilon}.

The first limit shows $\epsilon=e^{-\frac{1}{\gamma o(\gamma)}}$, and then as $\gamma \to 0$. 
\begin{equation*}
\begin{aligned}
\gamma \epsilon^{\gamma-1}&=\gamma e^{\frac{1}{\gamma o(\gamma)}-\frac{1}{o(\gamma)}}
&\to \gamma e^{\frac{1}{\gamma o(\gamma)}} \to \infty
\end{aligned}
\end{equation*}
Thus $\epsilon$ will not converge to $0$ as $\gamma \to 0$, and therefore, $\epsilon \geq C_0$ independent of $\gamma$.
\end{proof}
With the non-degeneracy property of the minimizer, we can prove the theorem by the method of compactness.
\begin{proof}
If not, then there exists $\gamma_k \to 0$ with $\{u_{\gamma_k}^j\}_{j=1}^{\infty}$ a sequence of minimizers of $J_{\gamma_k}$, and
\begin{equation}\label{gammakjto0}
\lim_{\gamma_k\to 0, j\to \infty}{L^{n}(\{u_{\gamma_k}^j=0\}\cap B_1^n)=0}.
\end{equation}
Without loss of generality, we assume $0$ is a common free boundary point and take blow-up limit at point $0$. Let $u_0^j=\lim_{\gamma_k \to 0}{u_{\gamma_k}^j}$. By the $\Gamma-$convergence of
\[
J_{\gamma}(u) \to J_0(u)=\int_{(B_1^{n+1})^+}{y^{\alpha}|\nabla u|^2}+\int_{B_1^n}{\chi_{\{u>0\}}},
\]
we know $\{u_0^j\}_{j=1}^{\infty}$ is a sequence of minimizers of $J_0(u)$. Then Lemma \ref{nondegeneracy} and \eqref{gammakjto0} will show
\[
\lim_{j \to \infty}{L^n(\{u_0^j=0\}\cap B_1^n)}=0
\]
which leads to contradiction, since in Theorem 1.3 in \cite{C3} the authors prove that in the fractional cavitation problem, near every free boundary point, the zero set has positive density. 

\end{proof}
\section{Flatness to regularity preliminaries and Main Theorem}\label{maintheorem}
In the following sections we apply the method introduced in \cite{D1} by De Silva, Savin and Sire to prove  the regularity of free boundary given flatenss conditon when $0<\gamma<\gamma_0$ (Theorem \ref{main}). \\

\subsection{Preliminaries}
First we give definitions and preliminaries of viscosity solutions to the free boundary problem and discuss the half-plane solution.\\

A point $X\in \mathbb{R}^{n+1}$ will be denoted by $X=(x,y)\in \mathbb{R}^n\times \mathbb{R}$. We also use the notation $x=(x',x_n)\in  \mathbb{R}^{n-1}\times \mathbb{R}$. For a function $g$ defined in $ (B^{n+1}_1)^+=\{ X\in \mathbb{R}^{n+1}, |X|<1,y>0\}$, we denote $\Omega^+(g)=\{g(x,0)>0\}\cap B_1^n$ as the positive set in $\mathbb{R}^n$, and $F(g)=\partial_{\mathbb{R}^n}\Omega^+(g)\cap B_1^n$ as the free boundary. We denote $\mathcal G(u)=\partial \{u>0\}\cap \partial B_1^n  \subset \partial B_1^{n+1}$ which is the boundary of the set $\partial \{u>0\}\cap \partial B_1^n$ in $\partial B_1^{n+1}$.
We consider the free boundary problem
\begin{equation}\label{fbeq}
\begin{cases} 
     div(y^{\alpha}\nabla g)=0 \ \ in \ \ (B^{n+1}_1)^+ ,\\
    \frac{\partial g}{\partial U}=1 \ \ on \ \ F(g),\\
    \lim_{y\to 0^+}{y^{\alpha}\partial_yg(x,y)}=\gamma g^{\gamma-1}(x) \ \ in \ \ \Omega^+(g).
   \end{cases}
\end{equation}
Here we denote
\[
\frac{\partial g}{\partial U}(x)=\lim_{t \to 0^+}{\frac{g(x+t\nu(x),0)}{t^{\beta}}}, x \in F(g),
\]
and $\nu(x)$ is the unit normal to $F(g)$ at $x$ towards the positive set $\Omega^+(g)$, and $U$ is defined as the following.\\

Consider $U(t,z)$ as the extension of $(t)_{+}^{\beta}$ to upper half plane, which satisfies $U(t,0)=(t)_{+}^{\beta}$, and $div(z^{\alpha}\nabla U(t,z))=0$ in $\{t\in \mathbb{R},z>0\}$.\\
Write $U(t,z)=r^{\beta}g(\theta)$, $r=\sqrt{t^2+z^2}>0$, $t=r\cos{\theta}$, $z=r\sin{\theta}$, and $\theta\in [0,\pi]$. Then the equation for $g(\theta)\geq 0$ is
\[
g''(\theta)+\alpha \cot \theta g'(\theta)+\beta(\alpha+\beta)g(\theta)=0
\]
with $g(\pi)=0, g(0)=1$, and $g(\theta)=1+\gamma(\sin \theta)^{2s}+o((\sin \theta)^{2s})$. The last equation is derived from $\lim_{z \to 0}{z^{\alpha}\partial_zU(t,z)}=\gamma U^{\gamma-1}(t,0)$ when $t>0$. The $(n+1)$-dimensional function $U(X)=U(x_n,z)$ is a solution with free boundary $\{x_n=0\}$. \\

\subsection{Viscosity solutions}
We now introduce the definition of viscosity solutions to \eqref{fbeq}. 
\begin{definition}
Given $g,v$ continuous, we say that $v$ touches $g$ by below (resp. above) at $X_0\in B_1^{n+1}$ if $g(X_0)=v(X_0)$ and
\[
g(X)\geq v(X) \ \ \text{ (resp. } g(X)\leq v(X))\text{ in a neighborhood}  \ O   \ \text{of} \ X_0.
\]
If this inequality is strict in $O\setminus \{X_0\}$, we say that $v$ touches $g$ strictly by below (resp. above).
\end{definition}

\begin{definition}
We say $v\in C((B_1^{n+1})^+)$ is a (strict) comparison subsolution to \eqref{fbeq} if $v$ is a non-negative function in $(B_1^{n+1})^+$ which is $C^2$ in the set where it is positive, and it satsfies
\begin{enumerate}[(i)]
\item  $div(y^{\alpha}\nabla v)\geq 0 \ \ in \ \ (B^{n+1}_1)^+$ .
\item  $F(v)$ is $C^2$ and if $x_0\in F(v)$ we have
\[
v(x,y)=aU((x-x_0)\cdot \nu(x_0),y)+o(|(x-x_0,y)|^{\beta}),\ \  \text{as} \ (x,y)\to (x_0,0),
\]
with $a\geq 1$, and $\nu(x_0)$ denotes the unit normal at $x_0$ to $F(v)$ towards the positive set $\Omega^+(v)$.\\
\item $ \lim_{y\to 0^+}{y^{\alpha}\partial_yv(x,y)}\geq \gamma v^{\gamma-1}(x) $.
\item Either $v$ satisfies (i) and (iii) strictly or $a>1$.
\end{enumerate}
Similarly one can define a comparison supersolution.
\end{definition}

\begin{definition}
We say that $g$ is a viscosity solution to \eqref{fbeq} if $g$ is a continuous non-negative function which satisfies
\begin{enumerate}[(i)]
\item  $g$ is locally $C^{1,1}$ in $(B_1^{n+1})^+$ and solves (in the viscosity sense)
\begin{equation*}
\begin{cases} 
div(y^{\alpha}\nabla g)= 0 \ \ in \ \ (B^{n+1}_1)^+ ,\\
\lim_{y\to 0^+}{y^{\alpha}\partial_g(x,y)}=\gamma g^{\gamma-1}(x,0) \ \ in \ \ \Omega^+(g).
\end{cases}
\end{equation*}
\item Any (strict) comparison subsolution (resp. supersolution) cannot touch $g$ by below (resp. by above) at a point $X_0=(x_0,0)\in F(g)$. 
\end{enumerate}
\end{definition}

\subsection{Comparison Principle}
We state the comparison principle for the problem \eqref{fbeq}. The proof is standard and can be found at Lemma 2.6 in \cite{D3}.\\
\begin{lemma}\label{comparison}
Let $g,v_t \in C(\overline{(B_1)^+})$ be respectively a solution and a family of subsolutions to \eqref{fbeq} with $0\leq t\leq 1$. Assume that
\begin{enumerate}[(i)]
\item $v_0\leq g$ in $\overline{(B_1)^+}$.
\item $v_t\leq g$ on $(\partial B_1^{n+1})^+$ for all $t\in [0,1]$.
\item $v_t<g$ on $\mathcal G(v_t)=\partial \{v_t>0\}\cap \partial B_1^n  \subset \partial B_1^{n+1}$. 
\item $v_t(x)$ is continuous in $(x,t)\in \overline{(B_1)^+}\times [0,1]$ and $\overline{\{v_t>0\}\cap B_1^n}$ is continunous in the Hausdorff metric.
\end{enumerate}
Then
\[
v_t \leq g \ \ \text{in} \ \ \overline{(B_1)^+} \ \  \text{for all} \ \ t \in [0,1]. 
\]
\end{lemma}
Then as a consequence of the lemma, we introduced the comparison principle used in this paper.
\begin{corollary}\label{comparisonp2}
Let $g$ be a solution to \eqref{fbeq} and let $v$ be a subsolution to \eqref{fbeq} in $(B_2^{n+1})^+$
which is strictly monotone in the $e_n$-direction in the set $\{v>0\}\cap B_2^{n+1}\cap \{y\geq 0\}$. Call
\[
v_t(X)=v(X+te_n), X\in B_1^+.
\]
Assume that for $-1\leq t_0\leq t_1\leq 1$,
\[
v_{t_0}\leq g \ \ \text{in} \ \ \overline{ (B_1^{n+1})^+},
\]
and
\[
v_{t_1}\leq g \ \ \text{on} \ \ \partial (B_1^{n+1})^+,\ \  v_{t_1}< g \ \ \text{on} \ \ \mathcal{G}(v_{t_1}).
\]
Then
\[
v_{t_1}\leq g \ \ \text{in} \ \ \overline{ (B_1^{n+1})^+}.
\]
\end{corollary}

\subsection{Main Theorem}

\begin{theorem}[Main Theorem]\label{main}
There exists $\gamma_0>0$ such that for each $0<\gamma<\gamma_0$, there exists a universal constant $\bar{\epsilon}>0$, such that if $g$ is a viscosity solution to \eqref{fbeq} satisfying the flatness condition
\[
\{x\in B_1^n, x_n\leq -\bar{\epsilon}\}\subset \{x\in B_1^n, g(x,0)=0\}\subset \{x\in B_1^n, x_n\leq \bar{\epsilon}\},
\]
then $F(g)$ is $C^{1,\theta}$ in $B_{1/2}^n$, with $\theta>0$ depending on $n,s$ and $\gamma$.
\end{theorem}
\begin{lemma}\label{flatness}
Assume $g_{\gamma}$ solves \eqref{fbeq}, and $U_{\gamma}$ is the half-plane solution. There exists $\gamma_0>0$ such that for each $0<\gamma<\gamma_0$, given any $\epsilon>0$, there exists $\bar{\epsilon}>0$ and $\delta>0$ depending on $\epsilon$ such that if 
\[
\{x\in B_1^n, x_n\leq -\bar{\epsilon}\}\subset \{x\in B_1^n, g_{\gamma}(x,0)=0\}\subset \{x\in B_1^n, x_n\leq \bar{\epsilon}\},
\]
then the rescaling $\delta^{-\beta}g_{\gamma}(\delta X)$ satisfies
\[
U_{\gamma}(X-\epsilon e_n)\leq \delta^{-\beta}g_{\gamma}(\delta X) \leq U_{\gamma}(X+\epsilon e_n) \ \ \text{in} \ \ B_1^n.
\]
\end{lemma}
\begin{proof}[Proof of Lemma \ref{flatness}]
We use the method of compactness since this lemma for case $\gamma=0$ is proved in Lemma 2.10 in \cite{D1}. Assume that there exists $\gamma_k \to 0$ such that the lemma does not hold for each $\gamma_k$. Then for each $\gamma_k$, there exists a sequence $\{g_{\gamma_k}^j\}_{j=1}^{\infty}$, $g_{\gamma_k}^j$ are solutions of \eqref{fbeq} with $\gamma=\gamma_k$, and a sequence $\{\epsilon_k^j\}_{j=1}^{\infty}$ with $\epsilon_k^j \to 0$ as $j \to \infty$ for each $k$, such that $g_{\gamma_k}^j$ satisfies the following condition with $\bar{\epsilon}_k^j \to 0$ as $j \to \infty$, 
\[
\{x\in B_1^n, x_n\leq -\bar{\epsilon}_k^j\}\subset \{x\in B_1^n, g_{\gamma_k}^j(x,0)=0\}\subset \{x\in B_1^n, x_n\leq \bar{\epsilon}_k^j\},
\]
but the conclusion does not hold for $\delta_k^j \to 0$ as $j \to \infty$.

Let $g_0^j=\lim_{\gamma_k\to 0}{g_{\gamma_k}^j}$, the limit exists since in \cite{Y1} the optimal $C^{\beta}$ estimates for the minimimzers are given, with $\beta=\frac{2s}{2-\gamma}>s$ and the $C^{\beta}$ norm does not blow-up as $\gamma \to 0$. And let $e_0^j=\lim_{k \to \infty}{e_k^j} \to 0$ as $j \to \infty$. The limit $U_0(X)=\lim_{\gamma_k\to 0}{U_{\gamma_k}}$ is the half-plane solution for the one-phase cavitation problem. In addition, we proved in Lemma \ref{nondegeneracy} that the minimizers are uniformly non-degenerate as $\gamma \to 0$. Then $\{u_0^j\}_{j=1}^{\infty}$ are the solutions of the case $\gamma=0$, and satisfy the flatness assumption with width $\sup_{k}{\bar{\epsilon}_{\gamma_k}^j} \to 0$ as $j \to \infty$, but the conclusion does not hold, which leads to a contradiction. 
\end{proof}

So from now on we may assume that 
\[
U(X-\epsilon e_n)\leq g(X)\leq U(X+\epsilon e_n) \ \ \text{in} \ \ B_1^n.
\]

The proof of Theorem \ref{main} is organized as follows. In Section \ref{Linearzed problem} we recall the $\epsilon-$domain variation of the solutions and the associated linearized equations. In Section \ref{HarnackInequality} we give the proof of a Harnack inequality and then we improve the flatness in Section \ref{improvement}. And in Section \ref{linearizedregularity} the regularity of the solutions to the  linearized equations are proved and we finish our proof of the main Theorem in Section \ref{pfofmain}. In the Appendix, several useful inequalities of the half-plane solution $U(t,z)$ are given.

\section{Linearized problem}\label{Linearzed problem}
In this section we recall the $\epsilon-$domain variation of the solution to \eqref{fbeq} and state the associated linearized problem, which are introduced in \cite{D1}. 

\subsection{The $\epsilon$-domain variations}

Let $P=\{X \in \mathbb{R}^{n+1}, x_n\leq 0,y=0\}$ and $L=\{X \in \mathbb{R}^{n+1}, x_n= 0,y=0\}$. To each $X\in \mathbb{R}^{n+1}\cap\{y\geq 0\}\setminus P$ we associate a set $\tilde{g}_{\epsilon}(X) \subset \mathbb{R}$ such that
\[
U(X)=g(X-\epsilon w e_n), \ \  \forall w\in \tilde{g}_{\epsilon}(X).
\]
We call $\tilde{g}_{\epsilon}$ the $\epsilon$-domain variation associated to $g$. And from now on we write $\tilde{g}_{\epsilon}(X)$ to denote any of the calues in this set, by abuse of notation. We claim the following: if $g$ satisfies
\begin{equation}\label{U_g_U}
U(X-\epsilon e_n)\leq g(X)\leq U(X+\epsilon e_n) \ \ \text{in} \ \ B_{\rho}^{n+1}\cap \{y\geq 0\},
\end{equation}
then
\[
 \tilde{g}_{\epsilon}(X)\in[-1,1].
\]
To prove this, same as in \cite{D3}, we let
\[
Y=X- \tilde{g}_{\epsilon}(X)e_n, \ \ X\in \mathbb{R}^{n+1}\cap\{y\geq 0\}\setminus P,
\]
then we can see
\[
U(Y-\epsilon e_n)\leq g(Y)=U(Y+\tilde{g}_{\epsilon}(X)e_n)\leq U(Y+\epsilon e_n),
\]
by our definition $U(X)=g(X-\epsilon \tilde{g}_{\epsilon}(X) e_n)>0$ and $U$ is strictly monotone in $e_n$-direction outside of $P$. And by \eqref{U_g_U}, for each $X\in B_{\rho-\epsilon}^{n+1}\cap \{y\geq 0\} \setminus P$, the set $\tilde{g}_{\epsilon}(X)$ is non-empty and there exists at least one valye such that
\[
U(X)=g(X-\epsilon \tilde{g}_{\epsilon}(X) e_n).
\]
And our claim follows by the continuity of $g(X-\delta \epsilon e_n)$, for $\delta\in [-1,1]$.

Moreover, if $g$ is strictly monotone in the $e_n$-direction, the $\tilde{g}_{\epsilon}(X)$ is single-valued.

The following lemma will be useful to obtain a comparisan principle.

\begin{lemma}\label{evariationcomp}
Let $g,v$ be respectively a solution and a subsolution to \eqref{fbeq} in $(B_{2}^{n+1})^+$. Assume that $g$ satisfies the flatness condition \eqref{U_g_U} in $(B_{2}^{n+1})^+$, that $v$ is strictly increasing in the $e_n$-direction in $\{v>0\}\cap B_{\rho}^{n+1} \cap \{y\geq 0\}$, and that $\tilde{v}_{\epsilon}$ is defined on $B_{2-\epsilon}^{n+1}\cap \{y\geq 0\} \setminus P$ with
\[
|\tilde{v}_{\epsilon}|\leq C<\infty.
\]
If 
\[
\tilde{v}_{\epsilon}+c \leq \tilde{g}_{\epsilon} \ \ \text{in} \ \ B_{3/2}^{n+1}\setminus \overline{B_{1/2}^{n+1}}\cap \{y\geq 0\} \setminus P,
\]
then we have
\[
\tilde{v}_{\epsilon}+c \leq \tilde{g}_{\epsilon} \ \ \text{on} \ \ B_{3/2}^{n+1}\cap \{y\geq 0\} \setminus P.
\]
\end{lemma}
The proof given in Lemma 3.2 in \cite{D3} is still valid since it only involves the comparison principle in Corollary \ref{comparisonp2} and the definition of $\tilde{g}_{\epsilon}$.\\

Given $\epsilon>0$ and a Lipschitz function $\tilde{\psi}$ defined on $B_{\rho}^{n+1}(Y)\cap \{y\geq 0\}$ with values in $[-1,1]$, there exists a unique function $\psi_{\epsilon}$ defined on  $B_{\rho-\epsilon}^{n+1}(Y)\cap \{y\geq 0\}$ such that
\[
U(X)=\psi_{\epsilon}(X-\epsilon \tilde{\psi}(X)e_n), X \in B_{\rho}^{n+1}(Y)\cap \{y\geq 0\}.
\]
And moreover $\psi_{\epsilon}$ is increasing in the $e_n$ direction. Thus, if $g$ satisfies the flatness condition \eqref{U_g_U} and $\tilde{\psi}$ is defined as above, then 
\[
\tilde{\psi}\leq \tilde{g}_{\epsilon} \ \ \text{in} \ \ B_{\rho}^{n+1}(Y)\cap \{y\geq 0\}\setminus P
\]
will lead to 
\begin{equation}\label{psiepsilon}
\psi_{\epsilon}\leq g  \ \ \text{in} \ \ B_{\rho-\epsilon}^{n+1}(Y)\cap \{y\geq 0\}.
\end{equation}

\subsection{The linearized problem}
We introduce here the linearized problem associated to \eqref{fbeq}. $U_n$ is the $x_n$-derivative of the funtion $U$. Given $w\in C((B_1^{n+1})^+)$ and $X_0=(x_0',0,0)$, we define
\[
|\nabla_rw|(X_0)=\lim_{(x_n,y)\to (0,0)}{\frac{w(x_0',x_n,y)-w(x_0',0,0)}{r}}, r^2=x_n^2+y^2.
\]
And the linearized problem associated to \eqref{fbeq} is
\begin{equation}\label{linearizedeq}
\begin{cases} 
div(y^{\alpha}\nabla(U_nw))= 0 \ \ in \ \ (B^{n+1}_1)^+ ,\\
|\nabla_rw|(X_0)=0 \ \  \text{on}\ \ B_1^n \cap L,\\
\lim_{y\to 0^+}{y^{\alpha}\partial_yw(x,y)}=0 \ \ \text{on} \ \ B_1^n\cap \{x_n>0\}.
\end{cases}
\end{equation}

The notion of the viscosity solution for this prolem is the following.

\begin{definition}
We say that $w$ is a solution to \eqref{linearizedeq} if $w \in C^{1,1}_{loc}( (B^{n+1}_1)^+)$ and it satisfies (in the viscosity sense)
\begin{enumerate}[(i)]
\item \begin{equation*}
\begin{cases} 
div(y^{\alpha}\nabla (U_nw))= 0 \ \ in \ \ (B^{n+1}_1)^+ ,\\
\lim_{y\to 0^+}{y^{\alpha}\partial_yw(x,y)}=0 \ \ \text{on} \ \  B_1^n\cap \{x_n>0\}.
\end{cases}
\end{equation*}
\item Let $\phi$ be continuous around $X_0=(x_0',0,0) \in B_1^n \cap L$ and satisfies
\[
\phi(X)=\phi(X_0)+a(X_0)\cdot(x'-x_0')+b(X_0)r+O(|x'-x_0'|^2+r^{1+\theta}), 
\]
for some $\theta>0$ and $b(X_0)\neq 0$.\\
If $b(X_0)>0$ then $\phi$ cannot touch $w$ by below at $X_0$, and if $b(X_0)<0$ then $\phi$ cannot touch $w$ by above at $X_0$.
\end{enumerate}
\end{definition}

\section{Harnack Inequality}\label{HarnackInequality}
In this section, we try to prove the following Harnack type inequality for solutions to the free boundary problem \eqref{fbeq}.

\begin{theorem}[Harnack Inequality]
\label{Harnack}
There exists $\bar{\epsilon}>0$ such that if $g$ solves \eqref{fbeq} and it satisfies
\[
U(X+\epsilon a_0 e_n)\leq g(X)\leq U(X+\epsilon b_0 e_n)\ \ \text{in} \ \ (B^{n+1}_{\rho}(X^*))^+,
\]
with $\epsilon(b_0-a_0)\leq \bar{\epsilon}\rho$, then
\[
U(X+\epsilon a_1 e_n)\leq g(X)\leq U(X+\epsilon b_1 e_n)\ \ \text{in} \ \ (B^{n+1}_{\eta\rho}(X^*))^+,
\]
with
\[
a_0\leq a_1\leq b_1\leq b_0, \ \ b_1-a_1\leq (1-\eta)(b_0-a_0),
\]
for a small universal constant $\eta$.
\end{theorem}

Let $g$ be a solution to \eqref{fbeq} which satisfies
\[
U(X-\epsilon e_n)\leq g(X)\leq U(X+\epsilon e_n) \ \ \text{in} \ \ (B_1^{n+1})^+.
\]
Let $A_{\epsilon}$ be the set
\[
A_{\epsilon}=\{(X,\tilde{g}_{\epsilon}(X)):X\in (B_{1-\epsilon}^{n+1})^+\} \subset \mathbb{R}^{n+1}\times[a_0,b_0].
\]
Since $\tilde{g}_{\epsilon}$ may be multi-valued, we mean all pairs $(X,\tilde{g}_{\epsilon}(X))$ for all possible values of $\tilde{g}_{\epsilon}$. An iterative argumemnt will give the following corollary of Theorem \ref{Harnack}.

\begin{corollary}\label{harnackcor}
If 
\[
U(X-\epsilon e_n)\leq g(X)\leq U(X+\epsilon e_n) \ \ \text{in} \ \ (B_1^{n+1})^+.
\]
with $\epsilon\leq \bar{\epsilon}/2$, given $m_0>0$ such that
\[
2\epsilon(1-\eta)^{m_0}\eta^{-m_0}\leq \bar{\epsilon},
\]
then the set $A_{\epsilon}\cap((B_{1/2}^{n+1})^+\times[-1,1])$ is above the graph of a function $y=a_{\epsilon}(X)$ and is below the graph of a function $y=b_{\epsilon}(X)$ with
\[
b_{\epsilon}-a_{\epsilon}\leq 2(1-\eta)^{m_0-1},
\]
and $a_{\epsilon},b_{\epsilon}$ having a modulus of continuity bounded by the H\"{o}lder function $At^{B}$ with $A,B$ depending only on $\eta$.
\end{corollary}

The proof of Harnack inequality follows as in the case when $\gamma=0$ in \cite{D1}. The key ingredient is the lemma below.

\begin{lemma}\label{RadialSub}
There exists $\bar{\epsilon}>0$ such that for all $0<\epsilon<\bar{\epsilon}$, if $g$ is a solution to \eqref{fbeq} such that
\[
g(X)\geq U(X) \ \ \text{in} \ \ (B_{1/2}^{n+1})^+,
\]
and at $\bar{X}\in (B_{1/8}^{n+1}(\frac{1}{4}e_n))^+$
\begin{equation}\label{gU}
g(\bar{X})\geq U(\bar{X}+\epsilon e_n),
\end{equation}
then
\begin{equation}\label{gUtau}
g(X)\geq U(X+\tau \epsilon e_n) \ \ \text{in} \ \  (B_{\delta}^{n+1})^+
\end{equation}
for universal constants $\tau,\delta$. Similarly, if 
\[
g(X)\leq U(X) \ \ \text{in} \ \ (B_{1/2}^{n+1})^+,
\]
and
\[
g(\bar{X})\leq U(\bar{X}-\epsilon e_n),
\]
then
\[
g(X)\leq U(X-\tau \epsilon e_n) \ \ \text{in} \ \  (B_{\delta}^{n+1})^+.
\]
\end{lemma}

There is a preliminary lemma.

\begin{lemma}\label{compgU}
Let $g\geq 0$ be the $C^{1,1}_{loc}$ in $(B_2^{n+1})^+$ and solves 
\[
div(y^{\alpha}\nabla g)=0 \ \ \text{in} \ \ (B_2^{n+1})^+,
\]
and let $\bar{X}=\frac{3}{2}e_n$. Assume that
\[
g\geq U \ \ \text{in} \ \  (B_2^{n+1})^+, \ \ g(\bar{X})-U(\bar{X})\geq \delta_0
\]
for some $\delta_0>0$. Then
\[
g\geq (1+c\delta_0)U \ \ \text{in} \ \ (B_1^{n+1})^+,
\]
for a small universal constant $c$. In particular, for any $0<\epsilon<2$,
\[
U(X+\epsilon e_n)\geq (1+c\epsilon)U(X)  \ \ \text{in} \ \ (B_1^{n+1})^+.
\]
\end{lemma}

The proof is slightly different since the boundary Harnack inequality of $U$ does not work. So instead we have the following proof.
\begin{proof}
We do an even extension of $U$ and $g$ with resepect to $\{y=0\}$, and let $g^*-U$ solves the following equation:
\begin{equation*}
\begin{cases} 
div(y^{\alpha}\nabla (g^*-U))= 0 \ \ \text{in} \ \ D=(B^{n+1}_{3/2})\setminus \{x_n<0,y=0\} ,\\
g^*-U=g-U\geq 0 \ \ \text{on} \ \ \partial B^{n+1}_{3/2},\\
g^*-U=0 \ \ \text{on} \ \  \{x_n<0,y=0\} .
\end{cases}
\end{equation*}
Then $g^*$ satisfies
\begin{equation*}
\begin{cases} 
div(y^{\alpha}\nabla g^*)= 0 \ \ \text{in} \ \ (B^{n+1}_{3/2})^+ ,\\
g^* \leq g \ \ \text{on} \ \ (\partial B^{n+1}_{3/2})^+,\\
g^*=0\leq g \ \ \text{on} \ \  \{x_n<0,y=0\},\\
\lim_{y \to 0}{ y^{\alpha}\partial_yg^*}\geq \lim_{y \to 0}{ y^{\alpha}\partial_yg} \ \ \text{on} \ \ \{x_n>0,y=0\}.
\end{cases}
\end{equation*}
The last inequality holds since
\[
\lim_{y \to 0}{ y^{\alpha}\partial_yg^*}=\lim_{y \to 0}{ y^{\alpha}\partial_yU}=\gamma U^{\gamma-1} \geq \gamma g^{\gamma-1}=\lim_{y \to 0}{ y^{\alpha}\partial_yg}.
\]
By maximum principle, $g^* \leq g$ in $(B^{n+1}_{3/2})^+$.
Let $\bar{X}=\frac{3}{2}e_n$, and $g(\bar{X})-U(\bar{X})\geq \delta_0$. Then since $g^*-U$ satisfies Harnack inequality,
\[
g^*-U=g-U\geq c_0\delta_0  \ \ \text{on} \ \ (\partial B^{n+1}_{3/2})^+\cap B_{1/4}^{n+1}(\bar{X}),
\]
and
\[
g^*(\tilde{X})- U(\tilde{X})\geq C_1\delta_0
\]
at some $\tilde{X}\in B_1^{n+1}\cap D$. Since $g^*-U$ satisfies boundary Harnack, 
\[
g^*(X)-U(X)\geq C_2\frac{g^*(\tilde{X})-U(\tilde{X})}{V(\tilde{X})}V(X) \ \ \text{in}\ \ (B_1^{n+1})^+.
\]
Here $V(X)$ solves
\[
div(y^{\alpha}\nabla V)=0 \ \ \text{in} \ \ D,
\]
and $V(X)=0$ on $\{x_n<0,y=0\}$. We can see $V(X)$ is the extension of $(x_n)_+^{s}$. Here we want to prove
\[
V(X)\geq C U(X)  \ \ \text{in} \ \ (B_1^{n+1})^+.
\]
We know
\begin{equation*}
\begin{cases} 
V(X)=|X|^sV(\frac{X}{|X|})=|X|^sh(\theta),\\
U(X)=|X|^{\beta}U(\frac{X}{|X|})=|X|^{\beta}g(\theta).
\end{cases}
\end{equation*}
and $\beta=\frac{2s}{2-\gamma}>s$. So we want to prove $\frac{h(\theta)}{g(\theta)}\geq C>0$ for $\theta \in [0,\pi]$. From Section 2.2 in \cite{D1}, $h(\theta)=(cos(\theta/2))^{2s}$. And from Section \ref{maintheorem}, $g(\theta)\geq 0$ solves the ODE
\begin{equation}\label{ODE}
g''(\theta)+\alpha \cot \theta g'(\theta)+\beta(\alpha+\beta)g(\theta)=0
\end{equation}
with $g(\pi)=0, g(0)=1$, and $g(\theta)=1+\gamma (sin\theta)^{2s}+o((sin\theta)^{2s})$ as $\theta \to 0$. So only problem occurs near $\theta=\pi$, where $h(\pi)=g(\pi)=0$. So
\begin{equation*}
\begin{aligned}
\lim_{\theta \to \pi}{\frac{h(\theta)}{g(\theta)}}&=\lim_{\theta \to \pi}{\frac{cos(\theta/2)^{2s}}{g(\theta)}}\\
&=\lim_{\theta \to \pi}{\frac{s(cos(\theta/2))^{2s-1}(-sin(\theta/2))}{g'(\theta)}}\\
&=\lim_{\theta \to \pi}{\frac{(-s)cos(\theta/2)^{2s-1}sin(\theta/2)sin(\theta)^{1-2s}}{g'(\theta)(sin\theta)^{\alpha}}}\\
&=\lim_{\theta \to \pi}{\frac{(-s)2^{1-2s}(sin(\theta/2))^{2-2s}}{g'(\theta)(sin\theta)^{\alpha}}}.
\end{aligned}
\end{equation*}
Our aim is to prove
\begin{equation}\label{g'sinalpha}
\begin{aligned}
\gamma \geq g'(\theta)(sin\theta)^{\alpha} \geq \gamma-C_0\beta(\alpha+\beta)\|g\|_{L^{\infty}}.
\end{aligned}
\end{equation}
Since
\begin{equation*}
\begin{aligned}
\lim_{\theta\to 0}{g'(\theta)(\sin{\theta})^{\alpha}}&=\lim_{\theta\to 0}{\frac{g(\theta)-g(0)}{\theta}(\sin{\theta})^{\alpha}}\\
&=\lim_{\theta\to 0}{\frac{\gamma \sin{\theta}^{2s}}{\theta}(\sin{\theta})^{\alpha}}\\
&=\gamma,
\end{aligned}
\end{equation*}
and $g$ solves the equation \eqref{ODE}, which is equivalent to
\begin{equation}\label{g'}
(g'(\theta)(sin\theta)^{\alpha})'=-\beta(\alpha+\beta)(\sin{\theta})^{\alpha}g(\theta),
\end{equation}
we can apply fundamental theorem of calculus and get
\[
g'(\theta)(\sin{\theta})^{\alpha}=\gamma-\beta(\alpha+\beta)\int_0^{\theta}{(\sin{\phi})^{\alpha}g(\phi)d\phi},
\]
so we need to prove $C_0=\int_0^1{(\sin{\theta})^{1-2s}d\theta}>0$ is a bounded number, which is ensured since $1-2s>-1$.
Now it is confirmed that
\[
\gamma \geq g'(\theta)(sin\theta)^{\alpha} \geq \gamma-C.
\]
So if $\tilde{C}\leq g'(\theta)(sin\theta)^{\alpha}\leq 0$ for some $\tilde{C}\leq 0$, then the limit will be a positive number (maybe positive infinity) and proof completed. If not, then it will contradicts our assumption that $\frac{h(\theta)}{g(\theta)}\geq 0$.\\
From above we prove that
\[
V(X)\geq C U(X)  \ \ \text{in} \ \ (B_1^{n+1})^+.
\]
Then the proof follows as
\[
g^*(X)-U(X)\geq C_2\frac{g^*(\tilde{X})-U(\tilde{X})}{V(\tilde{X})}V(X)\geq C\delta_0U(X) \ \ \text{in}\ \ (B_1^{n+1})^+,
\]
and
\[
g(X)\geq g^*(X)\geq (1+C\delta_0)U(X).
\]
\end{proof}

Proof of Lemma \ref{RadialSub} will use the following family of radial subsolutions. Let $R>0$ and denote
\[
V_R(t,z)=U(t,z)((n-1)\frac{t}{R}+1).
\]
Then set the (n+1)-dmensional function $v_R$ by rotating fuction $V_R$ around $(0,R,z)$.
\begin{equation}\label{defv_R}
v_R(X)=V_R(R-\sqrt{|x'|^2+(x_n-R)^2},z).
\end{equation}

\begin{proposition}\label{proptildev_R}
If $R$ is large enough, the function $v_R$ is a comparison subsolution to \eqref{fbeq} in $(B_2^{n+1})^+$ which is strictly monotone increasing in the $e_n$-direction. Moreover, there exists a function $\tilde{v}_R$ such that
\begin{equation}\label{tildev_R1}
U(X)=v_R(X-\tilde{v}_R(X)e_n) \ \ \text{in}\ \ (B_1^{n+1})^+
\end{equation}
and
\begin{equation}\label{tildev_R2}
|\tilde{v}_R(X)-\gamma_R(X)|\leq \frac{C}{R^2}|X|^2, \gamma_R(X)=-\frac{|x'|^2}{2R}+2(n-1)\frac{x_n r}{R},
\end{equation}
with $r=\sqrt{x_n^2+z^2}$ and $C$ universal.
\end{proposition}

\begin{proof}
Step 1. In this part we prove that  $v_R$ is a comparison subsolution and is strictly monotone increasing in the $e_n$-direction.\\
First, we need to prove $v_R$ is a strict subsolution to 
\begin{equation}\label{v_R}
div(z^{\alpha}\nabla v_R)=0 \ \ \text{in}\ \ (B_2^{n+1})^+.
\end{equation}
We can compute that
\begin{equation*}
\begin{aligned}
\Delta v_R&+\frac{\alpha}{z}\partial_z v_R\\
&=\Delta_{t,z}V_R(R-\rho,z)-\frac{n-1}{\rho}\partial_tV_R(R-\rho,z)+\frac{\alpha}{z}\partial_z V_R(R-\rho,z),
\end{aligned}
\end{equation*}
where $\rho=\sqrt{|x'|^2+(x_n-R)^2}$. Then for $(t,z) \in (\mathbb{R}^2)^+$,
\begin{equation*}
\begin{aligned}
\Delta_{t,z} v_R(t,z)&+\frac{\alpha}{z}\partial_z V_R(t,z)\\
&=(\partial_{tt}+\partial_{zz}) V_R(t,z)+\frac{\alpha}{z}\partial_z V_R(t,z)\\
&=\frac{2(n-1)}{R}\partial_tU+\partial_{tt}U(\frac{t(n-1)}{R}+1)+\partial_{zz}U(\frac{t(n-1)}{R}+1)+\frac{\alpha}{z}\partial_zU(\frac{t(n-1)}{R}+1)\\
&=\frac{2(n-1)}{R}\partial_tU(t,z),
\end{aligned}
\end{equation*}
and
\begin{equation}\label{dtV_R}
\partial_tV_R(t,z)=\partial_tU(t,z)(\frac{t(n-1)}{R}+1)+\frac{n-1}{R}U(t,z).
\end{equation}
To prove $v_R$ is a subsolution to \eqref{v_R} in $(B_2^{n+1})^+$, we need to show that
\[
\frac{2(n-1)}{R}\partial_tU-\frac{n-1}{\rho}[(\frac{t(n-1)}{R}+1)\partial_tU+\frac{n-1}{R}U] \geq 0
\]
evaluated at $(R-\rho,z)$. Set $t=R-\rho$, the inequality is reduced to
\begin{equation}\label{largeR}
[2(R-t)-R-(n-1)t]\partial_tU-(n-1)U\geq 0.
\end{equation}
To prove this, an inequality for function $U$ is required as
\begin{equation}\label{U_t/U}
r\frac{\partial_tU(t,z)}{U(t,z)}\geq C>0,
\end{equation}
with $r^2=t^2+z^2$. The proof of \eqref{U_t/U} is given in Section \ref{U_tU} in the Appendix.\\
Then we can show when $R$ is large enough, the inequality \eqref{largeR} is satisfied.\\

Next we want to prove that $v_R$ satisfies the free boundary condition. First observe that
\[
F(v_R)=\partial B_R^n(Re_n)\cap B_2^n(0),
\]
then we want to show
\begin{equation}\label{v_Rexp}
v_R(x,z)=aU(x_n,z)+o(|(x,z)|^{\beta}) \ \ \text{as} \ \ (x,z)\to (0,0),
\end{equation}
with $a\geq 1$. By the H\"{o}lder continuity of $U$ with exponent $\beta$, we can see
\[
|V_R(t,z)-V_R(t_0,z)|\leq C|t-t_0|^{\beta}\ \ \text{for} \ \ |t-t_0|\leq 1,
\]
thus for $(x,z)\in B_l^{n+1}$ for small $l>0$,
\[
|v_R(x,z)-V_R(x_n,z)|=|V_R(R-\rho,z)-V_R(x_n,z)|\leq C|R-\rho-x_n|^{\beta}\leq Cl^{2\beta}.
\]
Here we used
\[
R-\rho-x_n=-\frac{|x'|^2}{R-x_n+\rho}.
\]
Then it follows that
\begin{equation*}
\begin{aligned}
|v_R(x,z)-U(x_n,z)|&\leq |v_R(x,z)-V_R(x_n,z)|+|V_R(x_n,z)-U(x_n,z)|\\
&\leq Cl^{2\beta}+|U(x_n,z)|(n-1)\frac{|x_n|}{R}\\
&\leq Cl^{2\beta}+\tilde{C} l^{\beta+1}\\
&\leq \tilde{C}l^{2\beta},
\end{aligned}
\end{equation*}
since we can require $\gamma>0$ small enough such that $\beta=\frac{2s}{2-\gamma}\leq 1$. And this gives the desired expansion \eqref{v_Rexp} with $a=1$.\\

In the last part, we need to show that
\begin{equation}\label{visvr}
\begin{aligned}
\lim_{z \to 0}{z^{\alpha}\partial_z v_R(x,z)}\geq \gamma v_R^{\gamma-1}(x,0)
\end{aligned}
\end{equation}
for all $x \in \{v_R(x,0)>0\}\cap B_1^n$. From our definition of $v_{R}$, $x \in \{v_R(x,0)>0\}$ means $t=R-\rho>0$. We prove \eqref{visvr} by showing
\begin{equation*}
\begin{aligned}
\lim_{z \to 0}{z^{\alpha}\partial_z v_R(x,z)}&=\lim_{z \to 0}{z^{\alpha}\partial_z V_R(R-\rho,z)}\\
&=(\frac{(n-1)t}{R}+1)\lim_{z \to 0}{z^{\alpha}\partial_z U(R-\rho,z)}\\
&=(\frac{(n-1)t}{R}+1) \gamma U^{\gamma-1}(R-\rho,0)\\
&=(\frac{(n-1)t}{R}+1)^{2-\gamma}\gamma v_R^{\gamma-1}(x,0)\\
&\geq \gamma v_R^{\gamma-1}(x,0).
\end{aligned}
\end{equation*}
So we complete the proof that $v_R$ is a comparison subsolution to the equation \eqref{fbeq}.\\

And now, we show that $v_R$ is strictly monotone increasing in the $e_n$-direction. Since
\[
\partial_{x_n}v_R(x)=-\frac{x_n-R}{\rho}\partial_tV_R(R-\rho,z),
\]
so we only need to show $\partial_tV_R(R-\rho,z)>0$, which follows from \eqref{dtV_R} and \eqref{U_t/U}.\\

Step 2. In this part we state the existence of $\tilde{v}_R$ satisfying \eqref{tildev_R1} and \eqref{tildev_R2}.\\
First we want to show there exists unique $\tilde{t}$ such that
\begin{equation}\label{U=V_R}
U(t,z)=V_R(t+\tilde{t},z) \ \ \text{in} \ \ (B_1^2)^+
\end{equation}
and
\begin{equation}\label{tildet}
|\tilde{t}+\frac{2(n-1)tr}{R}|\leq \frac{\tilde{C}}{R^2}r^3,
\end{equation}
with $r^2=t^2+z^2$ and universal $\tilde{C}$. Since $V_R$ is strictly increasing in $t-$direction except $\{(t,0), t\leq 0\}$, so it suffices to show
\begin{equation}\label{V_RU}
V_R(t-\frac{2(n-1)tr}{R}-\frac{\tilde{C}}{R^2}r^3)<U(t,z)<V_R(t-\frac{2(n-1)tr}{R}+\frac{\tilde{C}}{R^2}r^3).
\end{equation}
To prove this, let
\[
\bar{t}=-\frac{2(n-1)tr}{R}-\frac{\tilde{C}}{R^2}r^3
\]
and then
\begin{equation}\label{taylorexp}
V_R(t+\bar{t},z)=V_R(t,z)+\bar{t}\partial_tV_R(t,z)+\frac{1}{2}E|\bar{t}|^2
\end{equation}
with
\[
|E|\leq |\partial_{tt}V_R(\tau,z)|, t+\bar{t}<\tau<t.
\]
Claim that
\begin{equation}\label{V_Rtt}
|\partial_{tt}V_R(\tau,z)|\leq \frac{C'}{r^2}U(t,z).
\end{equation}
And following is the proof.
\begin{equation*}
\begin{aligned}
\partial_{tt}V_R(\tau,z)&=\frac{n-1}{R}U_t+(\frac{(n-1)\tau}{R}+1)U_{tt}+\frac{n-1}{R}U_t\\
&=2\frac{n-1}{R}r^{\beta-1}U_t(\frac{\tau}{r},\frac{z}{r})+(\frac{(n-1)\tau}{R}+1)r^{\beta-2}U_{tt}(\frac{\tau}{r},\frac{z}{r}),
\end{aligned}
\end{equation*}
using $U$ is homogeneous of degree $\beta$. Since $\tau$ is between $t$ and $t+\bar{t}$, so $(\frac{\tau}{r},\frac{z}{r})\in B_{3/2}^+/B_{1/2}^+$. Here we claim that
\[
|\partial_{tt}U(\frac{\tau}{r},\frac{z}{r})|\leq K_1U(\frac{\tau}{r},\frac{z}{r}),
\]
and
\[
|\partial_{t}U(\frac{\tau}{r},\frac{z}{r})|\leq K_2U(\frac{\tau}{r},\frac{z}{r}).
\]
The proofs of these two inequalities are given in Section \ref{U_t/U<C} and Section \ref{U_ttUsz} in the Appendix. Then
\begin{equation*}
\begin{aligned}
|\partial_{tt}V_R(\tau,z)|&\leq 2\frac{n-1}{R}r^{\beta-1}K_2U(\frac{\tau}{r},\frac{z}{r})+(\frac{(n-1)\tau}{R}+1)r^{\beta-2}K_1U(\frac{\tau}{r},\frac{z}{r})\\
&\leq \bar{C}r^{\beta-2}U(\frac{\tau}{r},\frac{z}{r}).
\end{aligned}
\end{equation*}
Now what we want to prove is 
\begin{equation}\label{U_st}
U(\frac{\tau}{r},\frac{z}{r})\leq KU(\frac{t}{r},\frac{z}{r}),
\end{equation}
and then we can show
\[
|\partial_{tt}V_R(\tau,z)|\leq  \bar{C}r^{\beta-2}U(\frac{\tau}{r},\frac{z}{r}) \leq  \bar{C}Kr^{-2}U(t,z).
\]
In Section \ref{UKU} in the Appendix a proof of \eqref{U_st} is given, and our claim \eqref{V_Rtt} is now proved. Using \eqref{taylorexp} with the claim \eqref{V_Rtt}, we will be able to prove the lower bound in \eqref{V_RU} if we prove the following
\[
U(t,z)>V_R(t,z)+\bar{t}\partial_tV_R(t,z)+\frac{C'}{2r^2}U(t,z)|\bar{t}|^2,
\]
and it is equivalent to prove
\begin{equation*}
\begin{aligned}
U(t,z)&>U(t,z)(\frac{(n-1)t}{R}+1)+\bar{t}((\frac{(n-1)t}{R}+1)U_t(t,z)+\frac{n-1}{R}U(t,z))\\
&+\frac{C'}{2r^2}U(t,z)|\bar{t}|^2.
\end{aligned}
\end{equation*}
Divide both sides by $U$ and times $r$, it is equivalent to show
\[
\frac{(n-1)t}{R}r+\bar{t}(\frac{(n-1)r}{R}+[\frac{(n-1)t}{R}+1]r\frac{U_t}{U})+\frac{C'}{2r}|\bar{t}|^2<0.
\]
Plug in $\bar{t}=-\frac{2(n-1)tr}{R}-\frac{\tilde{C}}{R^2}r^3$, it is equivalent to show
\[
\bar{t}[\frac{(n-1)r}{R}-1/2+(r\frac{U_t}{U})(\frac{(n-1)t}{R}+1)]+\frac{C'}{2r}|\bar{t}|^2<\frac{\tilde{C}}{2R^2}r^3.
\]
By what we proved in \eqref{U_t/U}, and for $R$ large enough such that
\[
|\bar{t}|\leq Kr^2/R,
\]
we can show the above inequality is right for appropriate universal $\tilde{C}$ and $R$ large enough, thus lower bound in \eqref{V_RU} is proved.\\

To conclude, we use $R-\rho-x_n=-\frac{|x'|^2}{R-x_n+\rho}$ with $\rho=\sqrt{|x'|^2+(x_n-R)^2}$ to write
\[
v_R(X-\tilde{v}_Re_n)=V_R(R-\rho(\tilde{v}_R),z)=V_R(x_n-\tilde{v_R}-\frac{|x'|^2}{R-x_n+\tilde{v}_R+\rho(\tilde{v}_R)},z),
\]
with $\rho(\eta)=\sqrt{|x'|^2+(x_n-\eta-R)^2}$. In view of \eqref{U=V_R}, if there exists $\tilde{v}_R=\tilde{v}_R(X)$ such that
\begin{equation}\label{tildev_R}
-\tilde{v}_R-\frac{|x'|^2}{R-x_n+\tilde{v}_R+\rho(\tilde{v}_R)}=\tilde{t},
\end{equation}
then
\[
U(X)=v_R(X-\tilde{v}_Re_n),
\]
and by the strict monotonicity of $v_R$ in $e_n$ direction, $\tilde{v}_R$ must be unique. Thus, the proposition will be proved if we show that there exists $\tilde{v}_R$ satisfying \eqref{tildev_R} and such that
\[
|{\tilde{v}_R}(X)-\gamma_R(X)|\leq C\frac{|X|^2}{R^2}.
\]
To do so, we define
\[
f(\eta)=-\eta-\frac{|x'|^2}{R-x_n+\eta+\rho(\eta)}, -1\leq \eta\leq 1,
\]
and we show that 
\[
f(\gamma_R(X)+C\frac{|X|^2}{R^2})\leq \tilde{t}\leq f(\gamma_R(X)-C\frac{|X|^2}{R^2}),
\]
and using \eqref{tildet} we only need to prove that
\[
f(\gamma_R(X)+C\frac{|X|^2}{R^2}) \leq -\frac{2(n-1)x_nr}{R}-\tilde{C}\frac{r^3}{R^2},
\]
and
\[
f(\gamma_R(X)-C\frac{|X|^2}{R^2}) \geq -\frac{2(n-1)x_nr}{R}+\tilde{C}\frac{r^3}{R^2}.
\]
To prove the first (the second one follows similarly), we define
\[
\bar{\eta}=\gamma_R(X)+C\frac{|X|^2}{R^2},
\]
and from the definition of $f$ and $\gamma_R$, it is equivalent to show
\[
\frac{|x'|^2}{2R}-C\frac{|X|^2}{R^2}-\frac{|x'|^2}{R-x_n+\bar{\eta}+\rho(\bar{\eta})}\leq -\tilde{C}\frac{r^3}{R^2}.
\]
Since $-1\leq \bar{\eta} \leq 1$, so
\[
R-x_n+\bar{\eta}+\rho(\bar{\eta})\leq 2R+5
\]
and the inequality is reduced to
\[
-C\frac{|X|^2}{R^2}+\frac{|x'|^2}{R^2}\leq -\tilde{C}\frac{r^3}{R^2},
\]
which is satisfied as long as $C-\tilde{C}\geq 1$. 
\end{proof}

Then we can easily obtain the following Corollary.
\begin{corollary}\label{constant}
There exist $\delta,c_0,C_0,C_1$ universal constants such that

\begin{equation}\label{cor1}
v_R(X+\frac{c_0}{R}e_n)\leq (1+\frac{C_0}{R})U(X) \ \ \text{in} \ \ \overline{(B_1^{n+1})^+}/B_{1/4}
\end{equation}
with strict inequality on $F(v_R(X+\frac{c_0}{R}e_n))\cap (\overline{(B_1^{n+1})^+}/B_{1/4})$, and
\begin{equation}\label{cor2}
v_R(X+\frac{c_0}{R}e_n)\geq U(X+\frac{c_0}{2R}e_n) \ \ \text{in} \ \  (B_{\delta}^{n+1})^+,
\end{equation}
\begin{equation}\label{cor3}
v_R(X-\frac{C_1}{R}e_n)\leq U(X)  \ \ \text{in} \ \ \overline{(B_1^{n+1})^+}.
\end{equation}
\end{corollary}

And now we will start proving Lemma \ref{RadialSub}. We prove the first statement, and the second one follows similarly.
\begin{proof} 
In view of \eqref{gU},
\[
g(\bar{X})-U(\bar{X})\geq U(\bar{X}_\epsilon e_n)-U(\bar{X})=\partial_tU(\bar{X}+\lambda e_n)\epsilon \geq c\epsilon
\]
for $\lambda \in (0,\epsilon)$. From Lemma \ref{compgU}, we get
\begin{equation}\label{cor4}
g(X)\geq (1+c'\epsilon)U(X) \ \ \text{in} \ \ (B_{1/4}^{n+1})^+.
\end{equation}
Now let $R=\frac{C_0}{c'\epsilon}$, constants in Corollary \ref{constant}. Then for $\epsilon$ small enough, $v_R$ is a subsolution to \eqref{fbeq} in $(B_2^{n+1})^+$ which is monotone increasing in the $e_n$-direction and it also satisfies inequalities in Corollary \ref{constant}. We now apply the Comparison Principle stated in Corollary \ref{comparisonp2}. Let
\[
v_R^t(X)=v_R(X+te_n)
\]
ad according to \eqref{cor3},
\[
v_R^{t_0}\leq U\leq g  \ \ \text{in} \ \ (B_{1/4}^{n+1})^+,
\]
with $t_0=-C_1/R$. Moreover, from \eqref{cor1} to \eqref{cor4}, we get that for our choice of $R$,
\[
v_R^{t_1}\leq (1+c'\epsilon)U\leq g  \ \ \text{in} \ \ \partial (B_{1/4}^{n+1})^+,
\]
with $t_1=c_0/R$, with strict inequality on $F(v_R^{t_1})\cap \partial (B_{1/4}^{n+1})^+$. In particular,
\[
g>0 \ \  \text{on} \ \ \mathcal{G}(v_R^{t_1})\cap (B_{1/4}^{n+1})^+.
\]
Thus we can apply Comparison Principle to prove
\[
v_R^{t_1}\leq g  \ \ \text{in} \ \ (B_{1/4}^{n+1})^+.
\]
And thus from \eqref{cor2} we obtain
\[
U(X+\frac{c_1}{R}e_n)\leq v_R^{t_1}(X)\leq g (X) \ \ \text{in} \ \  (B_{\delta}^{n+1})^+,
\]
which is desired in \eqref{gUtau} with $\tau=\frac{c_1c'}{C_0}$.
\end{proof}

\section{Improvement of flatness}\label{improvement}
In this section we will show the proof of the improvement of flatness property for solutions to \eqref{fbeq}.
\begin{theorem}[Improvement of flatness]\label{Improvement}
There exists $\bar{\epsilon}>0$ and $\rho>0$ universal constants such that for all $0<\epsilon<\bar{\epsilon}$, if $g$ solves \eqref{fbeq} with $0\in F(g)$ and it satisfies
\begin{equation}\label{improv1}
U(X-\epsilon e_n)\leq g(X)\leq U(X+\epsilon e_n)  \ \ \text{in} \ \  (B_{1}^{n+1})^+,
\end{equation}
then
\begin{equation}\label{improv2}
U(x\cdot \nu-\epsilon \rho/2,z)\leq g(X)\leq U(x\cdot \nu+\epsilon \rho/2,z)  \ \ \text{in} \ \  (B_{\rho}^{n+1})^+,
\end{equation}
for some direction $\nu \in \mathbb{R}^n$, $|\nu|=1$.
\end{theorem}
The proof of Theorem \ref{Improvement} is divided into the next four lemmas.\\

The following lemma is the same as in Lemma 7.2 in \cite{D3} and its proof remained unchanged since it only depend on elementary properties related to the definition of $\tilde{g}_{\epsilon}$, and does not depend on the equation satisfied by $g$.\\
\begin{lemma}\label{L1}
Let $g$ be a solution to \eqref{fbeq} with $0\in F(g)$ and satisfying \eqref{improv1}. Assume that
\begin{equation}\label{}
a_0\cdot x'-\rho/4\leq \tilde{g}_{\epsilon}(X)\leq a_0\cdot x'+\rho/4  \ \ \text{in} \ \  (B_{2\rho}^{n+1})^+,
\end{equation}
for some $a_0\in \mathbb{R}^{n-1}$. Then if $\epsilon \leq \bar{\epsilon}(a_0,\rho)$, $g$ satisfies \eqref{improv2} in $(B_{\rho}^{n+1})^+$.
\end{lemma}

The next lemma follows immediately from Corollary \ref{harnackcor}.

\begin{lemma}\label{L2}
Let $\epsilon_k \to 0$ and let $g_k$ be a sequence of solutions to \eqref{fbeq} with $0 \in F(g_k)$ satisfying
\begin{equation}\label{g_kU}
U(X-\epsilon_k e_n)\leq g_k(X)\leq U(X+\epsilon_k e_n)  \ \ \text{in} \ \  (B_{1}^{n+1})^+.
\end{equation}
Denote by $\tilde{g}_k$ the $e_k$-domain variation of $g_k$. Then the sequence of sets
\[
A_k:=\{(X,\tilde{g}_k(X)):X\in (B_{1-\epsilon_k}^{n+1})^+\},
\]
has a subsequence that converges uniformly in Hausdorff distance in $ (B_{1/2}^{n+1})^+$ to the graph
\[
A_{\infty}:=\{(X,\tilde{g}_{\infty}(X)):X\in (B_{1/2}^{n+1})^+\},
\]
where $\tilde{g}_{\infty}$ is H\"{o}lder continuous.
\end{lemma}

\begin{lemma}\label{L3}
The limiting function satisfies $\tilde{g}_{\infty} \in C^{1,1}_{loc}(B_{1/2}^{n+1})^+$.
\end{lemma}

\begin{proof}
We fix a point $Y \in (B_{1/2}^{n+1})^+$, and let $\delta$ be the distance from $Y$ to $L=\{x_n=0,y=0\}$. It suffices to show that the function $\tilde{g}_{\epsilon}$ are uniformly $C^{1,1}$ in $B_{\delta/8}^{n+1}(Y)$. Since $g_{\epsilon}-U$ solves 
\[
div(y^{\alpha}\nabla(g_{\epsilon}-U))=0  \ \ \text{in} \ \  B_{\delta/2}^{n+1}(Y),
\]
we can see
\[
\|g_{\epsilon}-U\|_{C^{1,1}(B_{\delta/4}^{n+1}(Y))}\leq C\|g_{\epsilon}-U\|_{L^{\infty}(B_{\delta/2}^{n+1}(Y))}\leq C\epsilon,
\]
and  by implicit function theorem it follows as
\[
\|\tilde{g}_{\epsilon}\|_{C^{1,1}(B_{\delta/8}^{n+1}(Y))}\leq C,
\]
with constant $C$ depending on $Y$ and $\delta$.
\end{proof}

\begin{lemma}\label{ginfty}
The function $\tilde{g}_{\infty}$ solves the linearized problem \eqref{linearizedeq} in $(B_{1/2}^{n+1})^+$.
\end{lemma}

\begin{proof}
We start by showing that in the sense of viscosity, $U_n\tilde{g}_{\infty}$ satisfies 
\[
div(z^{\alpha}\nabla(U_n\tilde{g}_{\infty}))=0  \ \ \text{in} \ \  (B_{1/2}^{n+1})^+.
\]
Let $\tilde{\phi}$ be a $C^2$ function touching $\tilde{g}_{\infty}$ by below at $X_0=(x_0,z_0) \in  (B_{1/2}^{n+1})^+$, and we want to show that
\begin{equation}\label{U_ntildephi}
\Delta(U_n\tilde{\phi})(X_0)+\alpha\frac{\partial_z(U_n\tilde{\phi})(X_0)}{z_0}\leq 0.
\end{equation}
By Lemma \ref{L2}, the sequence $A_k$ converges uniformly to $A_{\infty}$, thus there exists a sequence of constants $c_k\to 0$ and a sequence of points $X_k\to X_0$ such that $\tilde{\phi}_k:=\tilde{\phi}+c_k$ touches $\tilde{g}_k$ by below at $X_k$ for $k$ large enough.\\
Define $\phi_k$ by below
\begin{equation}\label{phi_k}
\phi_k(X-\epsilon_k \tilde{\phi}_k(X)e_n)=U(X).
\end{equation}
Then according to \eqref{psiepsilon}, $\phi_k$ touches $g_k$ by below at $Y_k=X_k-\epsilon_k \tilde{\phi}_k(X_k)e_n$, for $k$ large enough. Thus, since $g_k$ solves 
\[
div(z^{\alpha}\nabla g_k)=0  \ \ \text{in} \ \  (B_{1}^{n+1})^+,
\]
it follows that
\begin{equation}\label{Deltaphi_kY_k}
\Delta(\phi_k)(Y_k)+\alpha\frac{\partial_{n+1}(\phi_k)(Y_k)}{z_k}\leq 0.
\end{equation}
Here we denote $\partial_{n+1}$ as the (n+1)-th derivative (same as $\partial_z$), and $z_k$ is the $n+1$-th coordinate of $Y_k$. Now we will compute $\Delta(\phi_k)(Y_k)$ and $\partial_{n+1}(\phi_k)(Y_k)$.\\
Since $\tilde{\phi}$ is smooth, for any $Y$ in a neighborhood of $Y_k$, there exists a unique $X=X(Y)$ such that
\begin{equation}\label{}
Y=X-\epsilon_k \tilde{\phi}_k(X)e_n.
\end{equation}
Thus \eqref{phi_k} reads as
\[
\phi_k(Y)=U(X(Y)),
\]
with $Y_i=X_i$ if $i \neq n$ and when $j \neq n$,
\[
\frac{\partial X_j}{\partial Y_i}=\delta_{ij}.
\]
Then
\begin{equation}\label{}
D_XY=I-\epsilon_kD_X(\tilde{\phi}_k(X)e_n),
\end{equation}
and
\begin{equation}\label{}
D_YX=I+\epsilon_kD_X(\tilde{\phi}e_n)+O(\epsilon_k^2),
\end{equation}
since
\[
\tilde{\phi}_k=\tilde{\phi}+c_k.
\]
It follows that 
\begin{equation}\label{X_nY_j}
\frac{\partial X_n}{\partial Y_j}=\delta_{jn}+\epsilon_k\partial_j\tilde{\phi}(X)+O(\epsilon_k^2).
\end{equation}
Then we can compute
\begin{equation}\label{Deltaphi_k}
\begin{aligned}
\Delta \phi_k(Y)&=U_n(X)\Delta X_n(Y)\\
&+\sum_{j\neq n}(U_{jj}(X)+2U_{jn}\frac{\partial X_n}{\partial Y_j})\\
&+U_{nn}(X)|\nabla X_n|^2(Y).
\end{aligned}
\end{equation}
By \eqref{X_nY_j}, we can calculate
\[
|\nabla X_n|^2(Y)=1+2\epsilon_k\partial_n\tilde{\phi}(X)+O(\epsilon_k^2),
\]
and
\begin{equation}\label{X_n^2}
\begin{aligned}
\frac{\partial^2 X_n}{\partial Y_j^2}&=\epsilon_k\sum_{i}{\partial_{ji}\tilde{\phi}\frac{\partial X_i}{\partial Y_j}}+O(\epsilon_k^2)\\
&=\epsilon_k\sum_{i \neq n}{\partial_{ji}\tilde{\phi}\delta_{ij}}+\epsilon_k\partial_{jn}\tilde{\phi}\frac{\partial X_n}{\partial Y_j}+O(\epsilon_k^2).
\end{aligned}
\end{equation}
Then
\begin{equation}\label{DeltaX_n}
\Delta X_n=\epsilon_k\Delta \tilde{\phi}+O(\epsilon_k^2).
\end{equation}
Using \eqref{DeltaX_n} and \eqref{X_n^2} in \eqref{Deltaphi_k}, we can get
\begin{equation}\label{Deltaphi_kY}
\Delta \phi_k(Y)=\Delta U(X)+\epsilon_kU_n\Delta \tilde{\phi}+2\epsilon_k\nabla \tilde{\phi}\cdot \nabla U_n+O(\epsilon_k^2)(U_nn+2\sum_{j \neq n}{U_{jn}}).
\end{equation}
And we can also calculate that
\begin{equation}\label{}
\begin{aligned}
(\phi_k)_{n+1}(Y)&=U_n(X)\frac{\partial X_n}{\partial Y_{n+1}}+U_z(X)\frac{\partial X_{n+1}}{\partial Y_{n+1}}\\
&=U_n(X)(\epsilon_k\partial_{n+1}\tilde{\phi}(X)+O(\epsilon_k^2))+U_z(X).
\end{aligned}
\end{equation}
Plug in \eqref{Deltaphi_kY}, and $\Delta U(X_k)+\frac{\alpha}{z}U_z(X_k)=0$ to \eqref{Deltaphi_kY_k}, we can calculate that
\begin{equation}\label{}
\begin{aligned}
\epsilon_k(U_n\Delta \tilde{\phi}+2\nabla \tilde{\phi}\nabla U_n+\Delta U_n \tilde{\phi}+\frac{\alpha}{z_k}U_n \partial_{n+1}\tilde{\phi}+\frac{\alpha}{z_k}(U_n)_z\tilde{\phi})+O(\epsilon_k^2)\leq 0,
\end{aligned}
\end{equation}
which means
\[
\Delta(U_n\tilde{\phi})(X_k)+\frac{\alpha}{z_k}\partial_z(U_n\tilde{\phi})(X_k)+O(\epsilon_k)\leq 0.
\]
And the desired \eqref{U_ntildephi} follows as $k \to \infty$.\\

The next step is to show that $\tilde{g}_{\infty}$ solves
\begin{equation}\label{dzg_infty}
\lim_{z\to 0}{z^{\alpha}\partial_z\tilde{g}_{\infty}}=0 \ \ \text{on} \ \ \{x_n>0\}\cap B_1^n.
\end{equation}
Since $\phi_k$ touches $g_k$ by below at $Y_k$ and $g_k$ solves \eqref{dzg_infty}, so
\[
\lim_{z\to 0}{z^{\alpha}\partial_z\phi_k(Y_k)}\geq \gamma \phi_k^{\gamma-1}(Y_k),
\]
and by the calculation in the previous part,
\[
\partial_z\phi_k(Y_k)=U_n(X)(\epsilon_k\partial_{n+1}\tilde{\phi}(X_k)+O(\epsilon_k^2))+U_z(X_k),
\]
therefore,
\begin{equation}\label{}
\begin{aligned}
\gamma \phi_k^{\gamma-1}(Y_k) & \leq \partial_z\phi_k(Y_k)\\
&=z^{\alpha}U_n \partial_{n+1}\tilde{\phi}(X_k) \epsilon_k+O(\epsilon_k^2)U_n(X_k)+z^{\alpha}\partial_zU(X_k).
\end{aligned}
\end{equation}
Since 
\[
\phi_k(Y_k)=U(X_k)
\]
as defined and $U$ satisfies
\[
\lim_{z\to 0}{z^{\alpha}\partial_zU}=\gamma U^{\gamma-1},
\]
we can show 
\[
\epsilon_k U_n z^{\alpha}\partial_{n+1}\tilde{\phi}(X_k)+O(\epsilon_k^2)U_n(X_k) \geq 0
\]
and thus
\[
z^{\alpha}\partial_{n+1}\tilde{\phi}(X_k)\geq 0.
\]
Here we use $U_n$ is strictly monotonuous increasing in the $e_n$-direction in $B_1^{n+1}\cap \{y\geq 0\} \setminus P$. Since $\tilde{\phi}_k=\tilde{\phi}+c_k$ touches $\tilde{g}_k$ by below, letting $k \to \infty$, we can prove that $\tilde{g}_{\infty}$ solves \eqref{dzg_infty} on $\{x_n>0\}\cap B_1^n$.\\

Then we want to show that $\tilde{g}_{\infty}$ solves 
\begin{equation}\label{}
|\nabla_r\tilde{g}_{\infty}|(X_0)=0, X_0=(x_0',0,0)\in B_{1/2}^n \cap L.
\end{equation}
 Assume by contradiction, there exists $\psi$ touching by below at $X_0$ and 
\[
\psi(X)=\psi(X_0)+a(X_0)(x'-x_0')+b(X_0)r+O(|x'-x_0'|^2+r^{1+l})
\]
for some $l>0$ and $b(X_0)>0$. Then there exists $\theta, \delta, \bar{r}$ and $Y'=(y_0',0,0)\in B_2$ depending on $\psi$ such that
\[
q(X)=\psi(X_0)-\frac{\theta}{2}|x'-y_0'|^2+2\theta(n-1)x_n r
\]
which is a second order polynomial touches $\psi$ by below at $X_0$, in a neighborhood $N_{\bar{r}}=\{|x'-x_0'|\leq \bar{r},r\leq \bar{r}\}$ of $X_0$. And $\psi-q\geq \delta>0$ on $N_{\bar{r}}/N_{\bar{r}/2}$. Then we can see
\[
\tilde{g}_{\infty}-q\geq \delta>0 \ \ \text{on}\ \ N_{\bar{r}}\setminus N_{\bar{r}/2},
\]
and
\[
\tilde{g}_{\infty}(X_0)-q(X_0)=0.
\]
In particular,
\[
|\tilde{g}_{\infty}(X_k)-q(X_k) \to 0, X_k\in N_{\bar{r}}\setminus \{x_n\leq 0,z=0\}, X_k\to X_0.
\]
Now choose $R_k=\frac{1}{\theta \epsilon_k}$ and and define
\[
w_k(X)=v_{R_k}(X_Y'+\epsilon_k\psi(X_0)e_n), Y=(y_0',0,0),
\]
with $v_R$ defined in \eqref{defv_R}. Then the $\epsilon_k$ domain variation of $w_k$ can be defined by
\[
w_k(X-\epsilon_k \tilde{w}_k(X)e_n)=U(X),
\]
and since $U$ is invariant in $x'$-direction, this is equivalent to
\[
v_{R_k}(X-Y'+\epsilon_k\psi(X_0)e_n-\epsilon_k \tilde{w}_k(X)e_n)=U(X-Y').
\]
Proposition \ref{proptildev_R} tells that
\[
\tilde{v}_{R_k}(X-Y')=\epsilon_k(\tilde{w}_k(X)-\psi(X_0)).
\]
Then we can conclude from \eqref{tildev_R2} that
\[
\tilde{w}_k(X)=q(X)+\theta^2\epsilon_kO(|X-Y'|^2),
\]
and hence
\[
|\tilde{w}_k-q|\leq C\epsilon_k  \ \ \text{on}\ \ N_{\bar{r}}\setminus \{x_n\leq 0,z=0\}
\]
Thus from the uniform convergence of $A_k$ to $A_{\infty}$, we get for $k$ large enough,
\begin{equation}\label{g_kw_k}
\tilde{g}_k-\tilde{w}_k\geq \delta/2  \ \ \text{on}\ \ (N_{\bar{r}}\setminus N_{\bar{r}/2})\setminus \{x_n\leq 0,z=0\}.
\end{equation}
And similarly we can get
\[
\tilde{g}_k(X_k)-\tilde{w}_k(X_k)\leq \delta/4,
\]
for some sequence $X_k \in N_{\bar{r}}\setminus \{x_n\leq 0,z=0\}$, and $X_k \to X_0$.\\
However from Lemma \ref{evariationcomp} and \eqref{g_kw_k}, we can see
\[
\tilde{g}_k-\tilde{w}_k\geq \delta/2 \ \ \text{on}\ \ N_{\bar{r}}/\{x_n\leq 0,z=0\}
\]
which leads to contradiction.\\
We complete the proof of Lemma \ref{ginfty} that $\tilde{g}_{\infty}$ solves the linearized problem \eqref{linearizedeq} in $(B_{1/2}^{n+1})^+$.
\end{proof}

We require regularity of the solutions to the linearized problem \eqref{linearizedeq} (in Section \ref{linearizedregularity}) to finish the proof of Theorem \ref{Improvement} in Section \ref{pfofmain}, and then the proof of the Main Theorem follows in that section.

\section{The regularity of linearized problem}\label{linearizedregularity}
In this section, our aim is to prove the regularity results for $w$ solving the linearized equation in the case $\gamma$ is small enough.
\begin{equation}\label{linearized2}
\begin{cases} 
div(y^{\alpha}\nabla((U_{\gamma})_nw))= 0 \ \ \text{in} \ \ (B^{n+1}_1)^+ ,\\
|\nabla_rw|(X_0)=0 \ \  \text{on}\ \ B_1^n \cap L,\\
\lim_{y\to 0^+}{y^{\alpha}\partial_yw(x,y)}=0 \ \ \text{on} \ \ B_1^n\cap \{x_n>0\}.
\end{cases}
\end{equation}
Here we denote the function $U_{\gamma}$ as the extension of $(x_n)_{+}^{\beta}$ to upper half space $(\mathbb{R}^{n+1})^+$, and the exponent $\beta=\frac{2s}{2-\gamma}$ depends on $\gamma$.\\
 
The following is the main theorem of this section.
\begin{theorem}\label{linearized}
There exists $\gamma_0>0$, such that for all $0<\gamma<\gamma_0$, the following regularity results hold.\\

Given a boundary data $\bar{h}\in C((\partial B_1^{n+1})^+)$, $|\bar{h}|\leq 1$, then there exists a unique classical solution $h$ to \eqref{linearized2} such that $h\in C(\overline{(B_1^{n+1})^+})$, $h=\bar{h}$ on $(\partial B_1^{n+1})^+$, and it satisfies
\begin{equation}\label{hX}
|h(X)-h(X_0)-a'\cdot(x'-x_0')|\leq C(|x'-x_0'|^2+r^{1+\theta}), X_0\in B_{1/2}^{n+1}\cap L,
\end{equation}
for universal constants $C,\theta$ and a vector $a'\in \mathbb{R}^{n-1}$ depending on $X_0$.
\end{theorem}

A corollary of the theorem above is what we require in the proof of the Theorem \ref{main}.

\begin{corollary}\label{linearizedcor}
There exists a universal constant $C$ such that if $w$ is a viscosity solution to \eqref{linearized2}, with
\[
-1\leq w(X)\leq 1 \ \ \text{in} \ \ (B^{n+1}_1)^+,
\]
then
\[
a_0\cdot x'-C|X|^{1+\theta}\leq w(X)-w(0)\leq a_0\cdot x'+C|X|^{1+\theta},
\]
for some vector $a_0\in \mathbb{R}^{n-1}$.
\end{corollary}

From Corollary \ref{linearizedcor}, there exists $\rho>0$,  if $w$ is a viscosity solution to \eqref{linearized2}, with $w(0)=0$ and
\[
-1\leq w(X)\leq 1 \ \ \text{in} \ \ (B^{n+1}_1)^+,
\]
then
\begin{equation}\label{defofrho}
a_0\cdot x'-\frac{1}{8}\rho \leq w(X)\leq a_0\cdot x'+\frac{1}{8}\rho, \ \ \text{in} \ \ (B_{2\rho}^{n+1})+ 
\end{equation}
for some vector $a_0\in \mathbb{R}^{n-1}$.

The proof of Theorem \ref{linearized} is based on method of compactness. In paper \cite{D1} section 6, Theorem 6.1 states the same results for the linearized problem of the limiting case $\gamma=0$. In the $\gamma=0$ case, $w$ solves
\begin{equation}\label{linearized0}
\begin{cases} 
div(y^{\alpha}\nabla((U_0)_nw))= 0 \ \ \text{in} \ \ (B^{n+1}_1)^+ ,\\
|\nabla_rw|(X_0)=0 \ \  \text{on}\ \ B_1^n \cap L,\\
\lim_{y\to 0^+}{y^{\alpha}\partial_y((U_0)_nw(x,y))}=0 \ \ \text{on} \ \ B_1^n\cap \{x_n>0\}.
\end{cases}
\end{equation}
with $U_0(X)=U_0(x_n,y)=(r^{1/2}\cos{(\theta/2}))^{2s}$, $r^2=x_n^2+y^2$. The regularity is stated same in Theorem \ref{linearized}. Our aim is to use method of compactness to prove Theorem \ref{linearized} for $0<\gamma<\gamma_0$ small enough.
\begin{proof}
If not, then there exists a sequence $\gamma_k\to0$ such that given boundary data $\bar{h}$ and $|\bar{h}|\leq 1$, $w_k$ solves \eqref{linearized2} for $\gamma=\gamma_k$ with boundary data $\bar{h}$, and for any $a'\in \mathbb{R}^{n-1},$ and for any $C>0, \theta>0$, there exists $X_k, \tilde{X}_k \in B_{1/2}^n \cap L$, such that
\[
|w_k(\tilde{X}_k)-w_k(X_k)-a'(x_k'-\tilde{x}_k')|>C(|x_k'-\tilde{x}_k'|^2+r^{1+\theta}).
\]
Consider the limit of its subsequence (denoted as $\gamma_k$, $X_k$, and $\tilde{X}_k$ as well), that $\tilde{X_k} \to \tilde{X}_0$, $X_k\to X_0$, and $w_k \to w_0$. Then $w_0=\bar{h}$ on $(\partial B_1^{n+1})^+$ and for any $a'\in \mathbb{R}^{n-1},$ and for any $C>0, \theta>0$,
\[
|w_0(\tilde{X}_0)-w_0(X_0)-a'(x_0'-\tilde{x}_0')|>C(|x_0'-\tilde{x}_0'|^2+r^{1+\theta}).
\]
Now we want to prove the limit $w_0$ solves \eqref{linearized0}, and then it leads to contradiction.\\
Let 
\[
J(w)=\int_{(B_1^{n+1})^+}{y^{\alpha}U_n^2|\nabla w|^2}dX.
\]
Then as proved in section 6 in \cite{D1}, the minimizer of the energy $J$ solves 
\begin{equation}\label{minJ}
div(y^{\alpha}U_n^2\nabla w)= 0 \ \ \text{in} \ \ (B^{n+1}_1)^+ ,\\
\end{equation}
and \eqref{minJ} is equivalent to
\begin{equation}\label{mineqn}
div(y^{\alpha}\nabla(U_nw))= 0 \ \ \text{in} \ \ (B^{n+1}_1)^+.\\
\end{equation}
Moreover, it is proved that if $w$ solves \eqref{mineqn}, and 
\[
\lim_{r \to 0}w_r(x',x_n,y)=b(x'), \ \ \text{on} \ \ L\cap B_1^n,
\]
then $w$ is a minimizer of $J(w)$ is equivalent to $b=0$. \\

Therefore, let $w_k$ be the solution to \eqref{linearized2} for $\gamma=\gamma_k$. Then $w_k$ is a minimizer of $J_{\gamma_k}(w)=\int_{(B_1^{n+1})^+}{y^{\alpha}(U_{\gamma_k})_n^2|\nabla w|^2}dX$, and $w_k$ satisfies
\[
\lim_{y\to 0^+}{y^{\alpha}\partial_y((U_{\gamma_k})_nw_k(x,y))}=w_k(x,0) \lim_{y\to 0^+}{y^{\alpha}\partial_y(U_{\gamma_k})_n}.
\]
This equality is derived from $\lim_{y\to 0^+}{y^{\alpha}\partial_yw_k(x,y)}=0$.\\

Let
\[
J_0(w)=\int_{(B_1^{n+1})^+}{y^{\alpha}(U_0)_n^2|\nabla w|^2}dX.
\]
Since we have $\lim_{\gamma \to 0}{U_{\gamma}} = U_0$ in $C((B^{n+1}_1)^+)$, thus if $w_k$ is a minimizer of $J_k(w)$, then $w_0=\lim_{\gamma_k \to 0}{w_{k}}$ is a minimizer of $J_0$. And by the convergence,
\[
(w_0)_r(x',x_n,y)=0
\]
Moreover, since $w_k \to w_0$,
\[
\lim_{y\to 0^+}{y^{\alpha}\partial_y((U_{\gamma_k})_nw_k(x,y))}=w_k(x,0) \lim_{y\to 0^+}{y^{\alpha}(U_{\gamma_k})_n},
\]
and
\[
\lim_{y\to 0^+}{y^{\alpha}\partial_y(\lim_{\gamma_k \to 0}{U_{\gamma_k}})_n}=\lim_{y\to 0^+}{y^{\alpha}\partial_y(U_{0})_n}=0,
\]
we can prove
\[
\lim_{y\to 0^+}{y^{\alpha}\partial_y((U_0)_nw_0(x,y))}=0.
\]
Therefore, we showed that the limit $w_0$ solves \eqref{linearized0}, which leads to contradiction. And Theorem \ref{linearized} is proved and Corollary \ref{linearizedcor} follows.
\end{proof}

\section{Proof of the main Theorem}\label{pfofmain}
In this section, we apply the regularity results of linearized problem \eqref{linearizedeq} to prove Theorem \ref{Improvement}. And then the proof of Main Theorem simply follows by Theorem \ref{Improvement} and Lemma \ref{flatness}.

\begin{proof}[Proof of Theorem \ref{Improvement}]
Let $\rho$ be the universal constant in \eqref{defofrho}, and assume by contradiction that there exists $
\epsilon_k \to 0$ and a sequence of solutions $g_k$ to \eqref{fbeq} such that $g_k$ satisfies
\begin{equation}\label{}
U(X-\epsilon_ke_n)\leq g_k(X)\leq U(X+\epsilon_ke_n) \ \ \text{in} \ \ (B^{n+1}_1)^+,
\end{equation}
but it does not satisfies the conclusion of the Theorem \ref{Improvement}.\\
Denote $\tilde{g}_k$ be the $\epsilon_k$-domain variation of $g_k$. Then by Lemma \ref{L2} the sequence of sets
\[
A_k:=\{(X,\tilde{g}_k(X)):X\in (B_{1-\epsilon_k}^{n+1})^+\},
\]
converges uniformly to
\[
A_{\infty}=\{(X,\tilde{g}_{\infty}(X)):X\in (B_{1/2}^{n+1})^+\},
\]
where $\tilde{g}_{\infty}$ is a H\"{o}lder continuous function. By Lemma \ref{ginfty}, the function $\tilde{g}_{\infty}$ solves the linearized equation \eqref{linearizedeq}, and hence by Corollary \ref{linearizedcor},
\[
a_0\cdot x'-\rho/8\leq \tilde{g}_{\infty} \leq a_0\cdot x'+\rho/8 \ \ \text{in} \ \ (B^{2\rho}_1)^+,
\]
with $a_0 \in \mathbb{R}^{n-1}$. From the uniform convergence of $A_k$ to $A_{\infty}$, we get that for all $k$ large enough,
\[
a_0\cdot x'-\rho/4 \leq \tilde{g}_{\infty} \leq a_0\cdot x'+\rho/4 \ \ \text{in} \ \ (B^{2\rho}_1)^+,
\]
and by Lemma \ref{L1}, $g_k$ satisfies \eqref{improv2}, which leads to a contradcition.
\end{proof}

\section{Appendix}
Let $U(t,z)=r^{\beta}g(\theta)\geq 0$, $r=\sqrt{t^2+z^2}$, $t=r\cos{\theta}$ and $z=r\sin{\theta}$, with $\theta \in [0,\pi]$. Since $div(z^{\alpha}\nabla U)=0$, and $\lim_{z\to 0}{z^{\alpha} \partial_zU(t,z)}=\gamma U^{\gamma-1}(t,0)$, so $g(\theta)$ solves the ODE
\begin{equation}\label{ODEg}
\begin{aligned}
g''(\theta)+\alpha \cot \theta g'(\theta)+\beta(\alpha+\beta)g(\theta)=0,
\end{aligned}
\end{equation}
with $g(\pi)=0$, $g(0)=1$, and $g(\theta)=1+\gamma (sin\theta)^{2s}+o((sin\theta)^{2s})$ as $\theta \to 0$. \\

\subsection{}\label{U_tU}  
In the first part, we try to prove the following inequalty:
\[
r\frac{\partial_tU(t,z)}{U(t,z)}\geq C>0.
\]
Calculate that
\[
\frac{U_t}{U}=\frac{1}{r}(\beta \cos{\theta}-\frac{g'(\theta)\sin{\theta}}{g(\theta)})=:\frac{1}{r}f(\theta).
\]
We define
\begin{equation}\label{fdef}
\begin{aligned}
f(\theta)=\beta \cos{\theta}-\frac{g'(\theta)\sin{\theta}}{g(\theta)},
\end{aligned}
\end{equation}
and then
\begin{equation}\label{f'def}
\begin{aligned}
f'(\theta)=\frac{1}{\sin{\theta}}[(f(\theta)-(\beta-s)\cos{\theta})^2+(\beta-s)^2\sin^2{\theta}-s^2].
\end{aligned}
\end{equation}
We can calculate $f(0)=\beta$ since
\begin{equation}\label{g'sin/gat0}
\begin{aligned}
\lim_{\theta \to 0}{\frac{g'(\theta)\sin{\theta}}{g(\theta)}}=\lim_{\theta \to 0}{\frac{g(\theta)-g(0)}{g(0)+\gamma (\sin{\theta})^{2s}}}=0,
\end{aligned}
\end{equation}
and $f(\pi)=2s-\beta>0$ since
\begin{equation}\label{g'sin/gatpi}
\begin{aligned}
\lim_{\theta \to \pi}{\frac{g'(\theta)\sin{\theta}}{g(\theta)}}&=\lim_{\theta \to \pi}{g'(\theta)(\sin{\theta})^{\alpha}\frac{(\sin{\theta})^{2s}}{g(\theta)}}\\
&=\lim_{\theta \to \pi}{g'(\theta)(\sin{\theta})^{\alpha}}\lim_{\theta \to \pi}{\frac{2s(\sin{\theta})^{2s-1}\cos{\theta}}{g'(\theta)}}\\
&=-2s.
\end{aligned}
\end{equation}
And to notice, $g'(\theta)(\sin{\theta})^{\alpha}$ is bounded and proof is given in \eqref{g'sinalpha}.
Also, we can calculate that $f'(0)=0$ and $f'(\pi)=0$ by 
\[
\lim_{\theta\to 0}{f'(\theta)}=\lim_{\theta\to 0}{\frac{2f f'-2(\beta-s)\cos{\theta} f'+2(\beta-s)\sin{\theta} f}{\cos{\theta}}}=\lim_{\theta\to 0}{2sf'(\theta)},
\]
and similarly
\[
\lim_{\theta\to \pi}{f'(\theta)}=\lim_{\theta\to 0}{-2sf'(\theta)}.
\]
Now we want to prove that $f(\theta)\geq C>0$ for $\theta\in [0,\pi]$. If not, then with the information of $f$ and $f'$ at the end points, there exists at least one $\theta_0 \in (0,\pi) $ such that
\begin{equation*}
\begin{cases} 
f'(\theta_0)=0,\\
f(\theta_0)\leq 0\\
f''(\theta_0)>0.
\end{cases}
\end{equation*}
Since $f'(\theta_0)=0$, 
\[
f(\theta_0)^2-2(\beta-s)\cos{\theta_0}f(\theta_0)+(\beta-s)^2-s^2=0,
\]
and thus
\[
f(\theta_0)=(\beta-s)\cos{\theta_0}\pm\sqrt{s^2-(\beta-s)^2\sin^2{\theta_0}}.
\]
If is the plus sign, then
\[
f(\theta_0)>(\beta-s)\cos{\theta_0}+(\beta-s)|\cos{\theta_0}|\geq 0
\]
which is not right. Thus
\[
f(\theta_0)=(\beta-s)\cos{\theta_0}-\sqrt{s^2-(\beta-s)^2\sin^2{\theta_0}}.
\]
Then we can calcuate $f''(\theta)$ at $\theta_0$, that
\[
f''(\theta)=\frac{(2ff'-2(\beta-s)\cos{\theta}f'+2(\beta-s)\sin{\theta}f)\sin{\theta}-(f' \sin{\theta})\cos{\theta}}{\sin^2{\theta}}.
\]
And when $\theta=\theta_0$,
\[
0<\sin^2{\theta}f''(\theta_0)=2(\beta-s)\sin^2{\theta_0}f(\theta_0)< 0,
\]
which leads to a contradiction.\\

\subsection{}\label{U_t/U<C}
In this section we try to prove
\[
|U_{t}(\frac{\tau}{r},\frac{z}{r})|\leq K_2U(\frac{\tau}{r},\frac{z}{r}).
\]
with $(\frac{\tau}{r},\frac{z}{r}) \in B_{3/2}^+/B_{1/2}^+$. Let $\theta=arctan(\frac{z}{\tau}) \in [0,\pi]$. Since $U$ is homomgeneous of degree $\beta$, we can see
\[
\frac{U_t(\frac{\tau}{r},\frac{z}{r})}{U(\frac{\tau}{r},\frac{z}{r})}=\frac{r}{\sqrt{\tau^2+z^2}}(\beta \cos{\theta}-\frac{g'(\theta)\sin{\theta}}{g(\theta)})\leq 2f(\theta)
\]
with 
\[
f(\theta)=\beta \cos{\theta}-\frac{g'(\theta)\sin{\theta}}{g(\theta)},
\]
which is the same definition as in \eqref{fdef}. As calculated in the previous section, $f(0)=\beta$, $f(\pi)=2s-\beta<\beta$, $f(\theta)\geq C>0$, and
\[
f'(\theta)=f'(\theta)=\frac{1}{\sin{\theta}}[(f(\theta)-(\beta-s)\cos{\theta})^2+(\beta-s)^2\sin^2{\theta}-s^2].
\]
Then if there exists $\theta_0$ such that $f(\theta_0)=+\infty$, then $f'=+\infty$ and will never be negative infinity at such $\theta_0$, which will lead to a contradiction of $f(0)=\beta$, $f(\pi)=2s-\beta<\beta$ and $\theta \in [0,\pi]$ which is a bounded interval. Therefore, there must exists an upper bound for $f(\theta)$ and then we can prove
\[
|\frac{U_t(\frac{\tau}{r},\frac{z}{r})}{U(\frac{\tau}{r},\frac{z}{r})}|\leq K_2.
\]

\subsection{}\label{U_ttU}
We try to prove 
\begin{equation}\label{U_tt/U}
|\frac{U_{tt}(t,z)}{U(t,z)}|\leq \frac{C(s,\gamma)}{r^2}.
\end{equation}
Write $U(t,z)=r^{\beta}g(\theta)$, where $t=r\cos{\theta}$, $z=r\sin{\theta}$ and $r=\sqrt{t^2+z^2}$. Then
\[
U_t=r^{\beta-2}(\beta g(\theta)t-g'(\theta)z), 
\]
and
\[
U_{tt}=r^{\beta-4}(((\beta^2-\beta)t^2+\beta z^2)g(\theta)+(2-2\beta)tzg'(\theta)+z^2g''(\theta)).
\]
Then
\[
r^2\frac{U_{tt}}{U}=(\beta^2-\beta)\cos^2{\theta}+\beta \sin^2{\theta}+\frac{g'(\theta)}{g(\theta)}(2-2\beta)\sin{\theta}\cos{\theta}+\frac{g''(\theta)}{g(\theta)}\sin^2{\theta}=:F(\theta).
\]
Since $div(z^{\alpha}\nabla U)=0$, so $g(\theta)$ solves
\[
g''(\theta)+\alpha \cot{\theta}g'(\theta)+\beta(\alpha+\beta)g(\theta)=0.
\]
Then we can replace $g''(\theta)$ in $F(\theta)$ and calculate
\begin{equation}\label{F(theta)}
\begin{aligned}
F(\theta)&=(\beta^2-\beta)\cos^2{\theta}+\beta \sin^2{\theta}+\frac{g'(\theta)}{g(\theta)}(2-2\beta)\sin{\theta}\cos{\theta}\\
&-\frac{\alpha \cot{\theta}g'(\theta)+\beta(\alpha+\beta)g(\theta)}{g(\theta)}\sin^2{\theta}\\
&=(\beta^2-\beta)\cos^2{\theta}+\beta(1-\alpha-\beta)\sin^2{\theta}+(2-2\beta-\alpha)\sin{\theta}\cos{\theta}\frac{g'(\theta)}{g(\theta)}.
\end{aligned}
\end{equation}
First,
\[
F(0)=\beta^2-\beta+\frac{g'(0)}{g(0)}\sin{\theta}(2-\alpha-2\beta)=\beta^2-\beta
\]
since $\frac{g'(0)}{g(0)}\sin{\theta}=0$. And
\[
F(\pi)=\beta^2-\beta-\lim_{\theta \to \pi}\frac{g'(\theta)}{g(\theta)}\sin{\theta}(2-\alpha-2\beta)=\beta^2-\beta+2s(2-\alpha-2\beta)=(2s-\beta)(2s-\beta+1).
\]
by
\[
\lim_{\theta \to \pi}\frac{g'(\theta)}{g(\theta)}\sin{\theta}=-2s.
\]
Notice that we require $\gamma>0$ small enough such that $\beta=\frac{2s}{2-\gamma}\leq 1$ in the proof of \eqref{v_Rexp}. So $F(0)\leq 0$ and $F(\pi)>0$. Then we calculate $F'(\theta)$:
\begin{equation*}
\begin{aligned}
F'(\theta)&=\beta(2-\alpha-2\beta)\sin{2\theta}+\frac{1}{2}(2-\alpha-2\beta)\frac{gg''\sin{2\theta}+2gg'\cos{2\theta}-(g')^2\sin{2\theta}}{g^2}\\
&=\frac{\beta}{2}(2-\alpha-2\beta)(2-\alpha-\beta)\sin{2\theta}+\frac{1}{2}(2-\alpha-2\beta)\frac{g'}{g}(-2+(2-2\alpha)\cos^2{\theta})\\
&-\frac{1}{2}(2-\alpha-2\beta)(\frac{g'}{g})^2\sin{2\theta}.
\end{aligned}
\end{equation*}
When $F'(\theta)=0$,
\[
\sin{2\theta}(\frac{g'}{g})^2+(2-(2-2\alpha)\cos^2{\theta})\frac{g'}{g}-\beta(2-\alpha-\beta)\sin{2\theta}=0.
\]
Then 
\begin{equation}\label{g'/g1}
\begin{aligned}
\frac{g'}{g}&=\frac{-(-1-(1-\alpha)\cos^2{\theta})\pm\sqrt{(-1-(1-\alpha)\cos^2{\theta})^2+\beta(2-\alpha-\beta)\sin^2{2\theta}}}{\sin{2\theta}}\\
&=\frac{-(-1-(1-\alpha)\cos^2{\theta})\pm\sqrt{L(\theta)}}{\sin{2\theta}}.
\end{aligned}
\end{equation}
And also, by \eqref{F(theta)}, we can calculate that
\begin{equation}\label{g'/g2}
\begin{aligned}
\frac{g'}{g}=2\frac{F(\theta)-(\beta^2-\beta)\cos^2{\theta}-\beta(2-\alpha-2\beta)\sin^2{\theta}}{(2-\alpha-2\beta)\sin{2\theta}}.
\end{aligned}
\end{equation}
Compare \eqref{g'/g1} and \eqref{g'/g2}, we can calculate that if $F'(\theta)=0$ at some $\theta_0 \in (0,\pi)$, then at $\theta_0$,
\[
F(\theta)=(\beta^2-\beta)\cos^2{\theta}+\beta(2-\alpha-2\beta)\sin^2{\theta}+\frac{1}{2}(2-\alpha-2\beta)[-(-1-(1-\alpha)\cos^2{\theta})\pm\sqrt{L(\theta)}]
\]
is a bounded number. With the conditions that $F(0)=\beta^2-\beta$ and $F(\pi)=(2s-\beta)(2s-\beta+1)$, we can prove that
\[
|F(\theta)|\leq C(s,\gamma),
\]
which is equivalent to
\[
|\frac{U_{tt}}{U}|\leq \frac{C(s,\gamma)}{r^2}.
\]

\subsection{}\label{U_ttUsz}
In this section, we try to prove 
\[
|\frac{U_{tt}(\frac{\tau}{r},\frac{z}{r})}{U(\frac{\tau}{r},\frac{z}{r})}|\leq K_1.
\]
Let $\theta=\arctan{\frac{z}{\tau}}\in [0,\pi]$, and we know $(\frac{\tau}{r},\frac{z}{r}) \in B_{3/2}^+/B_{1/2}^+$. Since $U$ is homogeneous of degree $\beta$, we can see
\[
|\frac{U_{tt}(\frac{\tau}{r},\frac{z}{r})}{U(\frac{\tau}{r},\frac{z}{r})}|=(\frac{r}{\sqrt{\tau^2+z^2}})^2|F(\theta)|\leq 4|F(\theta)|
\]
with $F(\theta)$ defined in \eqref{F(theta)}. Then using the results in the last section,
\[
|\frac{U_{tt}(\frac{\tau}{r},\frac{z}{r})}{U(\frac{\tau}{r},\frac{z}{r})}|\leq 4|F(\theta)|\leq 4C(s,r)=K_1.
\]

\subsection{}\label{UKU}
We try to prove if $\tau$ is between $t+\bar{t}$ and $t$, with
\[
\bar{t}=-\frac{2(n-1)tr}{R}-\frac{\tilde{C}}{R^2}r^3<0,
\]
then
\[
U(\frac{\tau}{r},\frac{z}{r})\leq KU(\frac{t}{r},\frac{z}{r}).
\]
Let $\theta_1=arccos(\frac{\tau}{\sqrt{\tau^2+z^2}})$ and $\theta_2=arccos(\frac{t}{r})$.
Since $g(\theta)\geq 0$ and $g(\theta)=0$ only when $\theta=\pi$, we only need to prove the inequality near $\theta_2=\pi$. Since $(\frac{\tau}{r},\frac{z}{r})\in B_{3/2}^+/B_{1/2}^+$, $t+\bar{t}\leq \tau\leq t$, and near $\theta_2=\pi$, $t<0$, we can see
\[
0<\pi-\theta_1\leq \pi-\theta_2.
\]
As calculated in \eqref{g'sin/gatpi},
\[
\lim_{\theta \to \pi}{\frac{g'(\theta)\sin{\theta}}{g(\theta)}}=-2s<0
\]
with $g\geq 0$ and $\sin{\theta}\geq 0$, we can see
\[
g'(\theta)\leq 0
\]
as $\theta \to \pi$.
Therefore when $\theta_1,\theta_2$ are close to $\pi$
\[
g(\theta_1)\leq g(\theta_2),
\]
and thus there exists $\bar{K}>0$ such that
\[
g(\theta_1)\leq Kg(\theta_2)
\]
for $\theta_1=arccos(\frac{\tau}{\sqrt{\tau^2+z^2}})$ and $\theta_2=arccos(\frac{t}{r})$. And thus there exists $K>0$ such that
\[
U(\frac{\tau}{r},\frac{z}{r})\leq (\frac{3}{2})^{\beta}g(\theta_1) \leq (\frac{3}{2})^{\beta}\bar{K}g(\theta_2)=KU(\frac{t}{r},\frac{z}{r}).
\]

\end{document}